\documentclass[11pt,a4paper]{article}
		\usepackage[margin=2.45cm]{geometry}
		\usepackage{amsmath,amssymb,amsfonts,bm,mathtools}
		\usepackage{amsthm}
\usepackage{booktabs,multirow,array,tabularx,longtable}
		\usepackage{graphicx,subcaption,float}
\numberwithin{equation}{section}
 \newtheorem{remark}{Remark}
 \newtheorem{theorem}{Theorem}[section]

		\usepackage{placeins,needspace}
		\floatstyle{ruled}
		\newfloat{algorithm}{tbp}{loa}
		\floatname{algorithm}{Algorithm}
		\usepackage{algpseudocode}
		\algrenewcommand\algorithmicrequire{\textbf{Input:}}
		\algrenewcommand\algorithmicensure{\textbf{Output:}}
		\algrenewcommand\algorithmicindent{1.0em}
		\algrenewcommand\alglinenumber[1]{\footnotesize #1:}
		\usepackage{tikz}
		\usetikzlibrary{arrows.meta,positioning,shapes.geometric,fit}
		\usepackage{xcolor}
		\usepackage{hyperref}
		\usepackage{cleveref}
		\usepackage{enumitem}
		\usepackage{siunitx}
		\usepackage{microtype}
      \usepackage{titling}

\pretitle{\begin{center}\LARGE\bfseries}
\posttitle{\par\end{center}\vskip 0.5em}

        \microtypesetup{nopatch=footnote}
		\usepackage{setspace}
		\usepackage{fancyhdr}

		\hypersetup{colorlinks=true,linkcolor=blue!60!black,citecolor=blue!60!black,urlcolor=blue!60!black}
		\usepackage{titlesec}
		\titlespacing*{\section}{0pt}{18pt}{9pt}
		\titlespacing*{\subsection}{0pt}{14pt}{6pt}
        \titleformat{\section}
  {\normalfont\Large\bfseries}
  {\thesection}{1em}{}
		
		\newcommand{\Ldata}{\mathcal{L}_{\mathrm{MC}}}
		\newcommand{\Lweak}{\mathcal{L}_{\mathrm{W}}}
		
		\newcommand{\Lphys}{\mathcal{L}_{\mathrm{phys}}}
		
		\newcommand{\J}{\bm{J}}
		\newcommand{\R}{\mathbb{R}}

		\newcommand{\safeincludegraphics}[2][]{%
		 \IfFileExists{#2}{\includegraphics[#1]{#2}}{%
		 \fbox{\parbox[c][4.7cm][c]{0.86\linewidth}{\centering
		 \textbf{Figure placeholder}\\[1mm]
		 {\ttfamily\footnotesize\detokenize{#2}}}}%
		 }%
		}
		
		\title{{Energetic Variational Deep Learning for Solving High-Dimensional Stationary Fokker-Planck Equations}\\
		 }
		\author{Fanze Kong \thanks{Department of Applied Mathematics, University of Washington, Seattle, WA 98195, USA; {\sf Email: fzkong@uw.edu}}; Zhouyuan Qin \thanks{School of Mathematics, Hunan University, Changsha, 410082, P. R. China, {\sf  Email: zhouyuan@hnu.edu.cn}} and  Shuangquan Xie\thanks{School of Mathematics, Hunan University, Changsha, 410082, P. R. China, {\sf  Email: xieshuangquan@hnu.edu.cn} } }
        
\begin{document}
		\maketitle
		
		\begin{abstract}
		This paper develops a deep learning framework for solving high-dimensional stationary Fokker-Planck equations by combining Monte Carlo weak supervision, energy-variational physics constraints, and quasi-Newton refinement. Monte Carlo statistics anchor the scale  of the stationary density, while a free-energy variational loss enforces equilibrium physics; the formulation also applies to stationary systems with pointwise
orthogonal rotational components and, using a family of test functions,
to systems with more general non-gradient drift. Numerical results demonstrate accurate and robust density approximations across a range of high-dimensional stationary Fokker–Planck problems, including rough-potential and noisy-data settings.
		
		\end{abstract}
		
		\section{Introduction}
        Stochastic differential equations (SDEs) constitute a fundamental framework for modeling dynamical systems subject to uncertainty and unresolved fluctuations, with applications ranging from option pricing and stochastic volatility modeling in mathematical finance to molecular dynamics and stochastic reaction systems in chemistry and biology.  Mathematically, such stochastic dynamics are typically represented by an It\^{o} SDE of the form 
        \begin{align}\label{SDE_main}
        dX_t = {\mathbf{b}}(X_t)dt + \sigma(X_t)dW_t,
        \end{align}
        where \(X_t\in\mathbb{R}^d\) denotes the state of the system, $\mathbf{b}:\mathbb{R}^d\rightarrow\mathbb{R}^d$ is the drift, \(\sigma:\mathbb{R}^d\to\mathbb{R}^{d\times d}\) is the diffusion coefficient matrix    and \(W_t\) is an \(d\)-dimensional Wiener process.  The Fokker-Planck equation provides a deterministic description of the evolution of the probability density associated with the stochastic process \(X_t\). When the stochastic dynamics approach a statistical steady state, the corresponding probability density becomes time-independent and satisfies the stationary Fokker-Planck equation.  Its solution gives the invariant probability distribution of \(X_t\) and, when the process is ergodic, characterizes the asymptotic distribution of the stochastic process.  In some cases, the stationary distribution can be characterized analytically. A classical example is Langevin dynamics in statistical physics, whose invariant distribution is given by the Gibbs measure. For more general stochastic dynamics, particularly those with non-gradient drift or state-dependent diffusion, the invariant distribution typically admits no closed-form representation and must be computed numerically by solving the corresponding stationary Fokker-Planck equation.

	    Classical numerical PDE methods, such as finite difference and finite element methods, have been extensively studied for Fokker-Planck equations and provide accurate and theoretically well-established approximations in low-dimensional settings \cite{masud2005application,kumar2009partition}.   Their applicability to high-dimensional problems, however, is severely limited by the curse of dimensionality, and the numerical solution of high-dimensional stationary Fokker-Planck equations remains challenging.   These challenges call for mesh-free numerical formulations suitable for high-dimensional state spaces.   Neural networks provide a flexible computational framework in this direction, with advances in both approximation theory and neural network-based solvers for high-dimensional PDEs \cite{barron1993universal,sirignano2018dgm,han2018solving,hutzenthaler2020proof,grohs2023proof}.   In particular, for stationary Fokker-Planck equations, Zhai et al. \cite{zhai2022deep} developed a data-driven deep learning method that combines Monte Carlo samples with PDE constraints, using the sampled data as a reference to eliminate the reliance on explicit boundary conditions and a neural network to represent the stationary density in a mesh-free manner.   The PDE constraint in this approach is imposed through a strong-form residual. Such a residual-based formulation is central to physics-informed neural networks (PINNs) \cite{raissi2019physics}, a pioneering framework for incorporating differential equations into neural network training.  Despite their success, PINNs can face optimization difficulties associated with higher-order derivatives and sensitivity to noisy data. This motivates the use of weak and variational formulations, which incorporate the governing equations in an integral form and reduce the order of derivatives required in the loss function.  Several neural PDE solvers have been developed along this direction, including the Deep Ritz method \cite{yu2018deep},  variational physics-informed neural networks (VPINNs) \cite{kharazmi2021hp}, with variational formulations further extended to PDE-governed inverse problems \cite{kong2025moment,kong2025moment1}.  Variational formulations reduce the order of derivatives through integration by parts and provide a more robust and computationally efficient alternative to strong-form residual minimization.  Our variational formulation originates from the energetic variational approach (EnVarA) \cite{onsager1931reciprocal1,onsager1931reciprocal2}, a thermodynamically consistent framework for modeling systems involving both conservative and dissipative dynamics. At its core, EnVarA characterizes the system evolution through the energy-dissipation law
\[
\frac{\mathrm d}{\mathrm d t}E^{\mathrm{total}}=-\Delta\leq 0,\]
where \(E^{\mathrm{total}}\) denotes the total energy of the system, consisting of the kinetic energy and the Helmholtz free energy, and \(\Delta\) represents the rate of energy dissipation.

        In this paper, motivated by the robustness of the variational formulation, we develop a framework that combines Monte Carlo (MC) data with a weak-form loss for high-dimensional stationary Fokker-Planck equations, integrating statistical information with the underlying physical structure to improve accuracy and training stability. Our approach employs a progressive training strategy, starting with Monte Carlo-based pretraining, followed by joint optimization of the Monte Carlo and variational losses and the  Limited-memory Broyden–Fletcher–Goldfarb–Shanno (L-BFGS)-based fine-tuning to further improve the accuracy of the solution.  In particular, to address the scale imbalance between the Monte Carlo and variational losses during joint optimization, we reformulate the variational loss in a degree-zero homogeneous form.	The main contributions of the proposed framework are summarized as follows:
		\begin{itemize}
		 \item A weak variational loss is derived from the free-energy dissipation structure for equilibrium steady states, enhancing robustness to rough potentials and noisy data. Dividing the loss by the total mass further makes it scale-invariant and positively homogeneous of degree zero, making its scale comparable to that of the Monte Carlo loss during optimization.
         
		 \item A weak-form learning framework combining Monte Carlo data and an integral variational physics loss is developed to reduce training-point requirements and computational cost, enabling efficient learning of high-dimensional steady-state problems.
		 \item The proposed framework applies to both equilibrium steady states and nonequilibrium steady states with pointwise-orthogonal potential and rotational components, extending its scope beyond equilibrium systems.
		\end{itemize}

        The rest of this paper is organized as follows. Section \ref{sec:formulation} presents the mathematical formulation, starting from the equilibrium condition of the free-energy dissipation law, followed by the derivation of the variational loss and its numerical discretization.  Section \ref{sec:method} is devoted to the proposed three-stage learning framework, comprising Monte Carlo pretraining, joint optimization of  Monte Carlo and variational losses, and L-BFGS-based fine-tuning.   The accuracy, robustness, and high-dimensional performance of the proposed method are examined through numerical experiments in Section \ref{sec:numerics}. It also includes ablation studies comparing different combinations of Monte Carlo data, strong- and weak-form losses, and Adam and L-BFGS optimization strategies.  Section \ref{sec:conclusion} concludes the paper with a summary of the main findings and directions for future research.

        \section{Mathematical Formulation}		\label{sec:formulation}
	   For simplicity, we assume an isotropic diffusion coefficient
    \(\sigma=\sqrt{2D}\,I_d\) in \eqref{SDE_main}, with constant \(D>0\). Then the SDE becomes

		\begin{equation}
		\mathrm dX_t=\bm b(X_t)\,\mathrm dt+\sqrt{2D}\,\mathrm dW_t,
		\qquad X_t\in\R^d.
		\label{eq:sde}
		\end{equation}
         Let \(m(x,t)\) denote the probability density of the stochastic state \(X_t\). Then, 
       $m(x,t)$ satisfies the following Fokker-Planck equation
		\begin{equation}
		\partial_t m=D\Delta m-\nabla\cdot(\bm b m)
		 =-\nabla\cdot\J[m],
		\qquad
		\J[m]:=\bm b m-D\nabla m,
		\label{eq:fp}
		\end{equation}
        where \(\J\) denotes the probability flux.
		Hence, the stationary density $m_\infty$ is the solution to
		\begin{equation}
		\nabla\cdot\J[m_\infty]=0.
		\label{eq:steady_general}
		\end{equation}

                We consider three cases according to the structure of the drift and the available stationary relation: the equilibrium zero-current case, the pointwise-orthogonal nonequilibrium case, and the general nonequilibrium case handled through a test-function family. For the gradient case, $\bm b=-\nabla V,$ corresponding to an equilibrium steady state, which satisfies $\J[m_\infty]=0$.
        In the non-gradient case, corresponding to a nonequilibrium steady state, we assume the drift is decomposed as
\begin{align}
\bm b=-\nabla V+\bm R,
\qquad
\bm R\cdot\nabla V=0
\quad\text{pointwise},
\label{non-equi-condition}
\end{align}
where \(V\) denotes the potential and \(\bm R\) the rotational component of the drift.
	 
Next, we derive the variational loss within the deep learning framework  for both cases from the free-energy dissipation structure. First of all, we have the free energy is defined as
		\begin{equation}
		\mathcal F[m]
		=D\int_{\mathbb R^d} m\log m\,\mathrm dx
		 +\int_{\mathbb R^d} V m\,\mathrm dx.
		\label{eq:free_energy}
		\end{equation}

		\noindent\textbf{Case I: Equilibrium Steady State.} 
      In light of ${\bf b}=-\nabla V$, a direct calculation gives
\begin{align}
\frac{\mathrm d}{\mathrm dt}\mathcal F[m(t)]
&=\int_{\mathbb R^d} (D\log m+V)\,
\nabla\cdot\left(m\nabla(D\log m+V)\right)\,\mathrm dx\nonumber \\
&=-\int_{\mathbb R^d}  m\left|\nabla(D\log m+V)\right|^2\,\mathrm dx\nonumber \\
&=-\int_{\mathbb R^d} 
\frac{\left|D\nabla m-\bm b m\right|^2}{m}\,\mathrm dx
\le 0.
\label{eq:dissipation_identity}
\end{align}
Moreover, at an equilibrium steady state, the dissipation vanishes, and \eqref{eq:dissipation_identity} yields
\begin{align}
\int_{\mathbb R^d}
\frac{|D\nabla m_\infty-\bm b m_\infty|^2}{m_\infty}\,dx=0.
\label{eq:zero_flux}
\end{align}
We approximate the stationary density \(m_\infty\) by a neural network \(u_\theta\) with parameters \(\theta\). The zero-flux condition \eqref{eq:zero_flux} then yields the variational physics loss
		\begin{equation}
		\Lweak(\theta)=
		\frac{\displaystyle\int_\Omega
		\frac{\left|D\nabla u_\theta-\bm b u_\theta\right|^2}
		{u_\varepsilon}\,\mathrm dx}
		{\displaystyle\int_\Omega u_\theta\,\mathrm dx}, ~~u_\varepsilon=\max(u_\theta,\varepsilon_u),
		\qquad
		\varepsilon_u=10^{-12},
		\label{eq:weak_loss}
		\end{equation}
        where  \(\Omega\subset\mathbb{R}^d\) denotes a truncated computational domain approximating the full state space \(\mathbb{R}^d\).  Here, the normalization by \(\int_\Omega u_\theta\,\mathrm dx\) ensures scale invariance under \(u_\theta\mapsto c u_\theta\), \(c>0\), which facilitates the balance between the physics and Monte Carlo data losses and the regularized density \(u_\varepsilon=\max(u_\theta,\varepsilon_u)\) is used only to avoid numerical instability near \(u_\theta=0\).

				\medskip
		\noindent\textbf{Case II: Nonequilibrium Steady State with Pointwise Orthogonal Drift Decomposition.}
Noting that $\bf b$ satisfies (\ref{non-equi-condition}), we compute the evolution of the free energy and obtain
        \begin{align}
\frac{\mathrm d}{\mathrm dt}\mathcal F[m(t)]
&=
\int_\Omega (D\log m+V)
\left[
\nabla\cdot\left(m\nabla(D\log m+V)\right)
-\nabla\cdot(m\bm R)
\right]\mathrm dx
\nonumber\\
&=
-\int_\Omega
m\left|\nabla(D\log m+V)\right|^2\,\mathrm dx
+
\int_\Omega
m\bm R\cdot\nabla(D\log m+V)\,\mathrm dx
\nonumber\\
&=-\int_\Omega
m\left|\nabla(D\log m+V)\right|^2\,\mathrm dx+
D\int_\Omega \bm R\cdot\nabla m\,dx
+\int_\Omega m\bm R\cdot\nabla V\,dx\nonumber\\
&=
-\int_\Omega
m\left|\nabla(D\log m+V)\right|^2\,\mathrm dx
\le 0,
\label{eq:dissipation_noneq}
\end{align}
where we have used $\bm { R}\cdot \nabla V =0 $ and $\nabla\cdot \bm{R}=0.$

Similarly, we approximate \(m_\infty\) by a neural network \(u_\theta\).  In case II, the dissipation identity \eqref{eq:dissipation_noneq} yields the variational physics loss
		\begin{equation}
		\Lweak(\theta)=
		\frac{\displaystyle\int_\Omega
		\frac{
		\left|D\nabla u_\theta
		+u_\theta\nabla V\right|^2
		}
		{u_\varepsilon}\,\mathrm dx}
		{\displaystyle\int_\Omega u_\theta\,\mathrm dx},\label{eq:orthogonal_weak_loss}
		\end{equation}
        where $u_{\varepsilon}$ is defined in (\ref{eq:weak_loss}).

		\medskip
\noindent\textbf{Case III: General nonequilibrium steady states.}
Write $\J[u]=\bm b\,u-D\nabla u$.
A nonequilibrium steady state may have a nonzero
current, but it satisfies
$\nabla\cdot\J[m_\infty]=0$.
By using integration by parts, one has
\[
  \int_\Omega
  \J[m_\infty]\cdot\nabla\varphi\,\mathrm dx=0
  \qquad
  \bigl(\varphi\in H_0^1(\Omega)\bigr).
\]

Choose nonzero test functions
$\varphi_1,\ldots,\varphi_K\in H_0^1(\Omega)$
and set
$\widehat\varphi_k
=\varphi_k/\|\nabla\varphi_k\|_{L^2(\Omega)}$.
For a positive approximation $u_\theta$, define
\begin{equation}
  r_k(\theta)
  =
  \frac{\displaystyle
    \int_\Omega
    \J[u_\theta]\cdot\nabla\widehat\varphi_k
    \,\mathrm dx}
  {\displaystyle
    \int_\Omega u_\theta\,\mathrm dx},
  \qquad k=1,\ldots,K.
  \label{eq:ness_test_moments}
\end{equation}
We measure the error in these weak equations by
\begin{equation*}
  \Lweak^{\rm TF}(\theta)
  =
  \frac1K\sum_{k=1}^K\omega_k r_k(\theta)^2,
\end{equation*}
where $\omega_k>0$ and
$K^{-1}\sum_{k=1}^K\omega_k=1$.
For finite $K$, this loss tests the stationary
equation against the selected $K$ functions.  The constant $c>0$ cancels between the numerator
and denominator of \eqref{eq:ness_test_moments} , so $u_\theta$ and $cu_\theta$
give the same weak loss.

The loss uses the original drift $\bm b$ and requires
only $u_\theta$ and $\nabla u_\theta$.
It allows a nonzero stationary current and does not
use the pointwise orthogonality condition in Case II.
In the numerical experiment below, we choose the
$K$ sine functions corresponding to the smallest
eigenvalues of the Dirichlet Laplacian on the
computational domain.
We set $\omega_k=1$, so the loss is the average
of the squared quantities $r_k(\theta)^2$.


\subsection{Discretization}
        In this subsection, we discretize the SDE and the proposed loss functions for subsequent numerical training.  

        Equation \eqref{eq:sde} is discretized using the Euler–Maruyama scheme. After a burn-in period \(T_{\rm burn}\) to reduce the influence of the initial condition, trajectory states are sampled at fixed time intervals. The samples are projected onto a regular histogram grid to obtain the empirical stationary-density values \(\widehat m_i\), from which \(N_{\rm ref}\) data pairs \(\{(x_i,\widehat m_i)\}_{i=1}^{N_{\rm ref}}\) are selected for training.  Here, \(T_{\rm burn}\) is chosen sufficiently large for the trajectory to reach the stationary regime.

First of all, we define  the standard pointwise mean squared error (MSE) of MC data as
		\begin{equation}
		\Ldata(\theta)
		=\frac{1}{N_{\rm ref}}\sum_{i=1}^{N_{\rm ref}}
		\left(u_\theta(x_i)-\widehat m_i\right)^2.
		\label{eq:data_mse}
		\end{equation}
	Denote the \(i\)-th reference bin by
		\[
		B_i=\prod_{k=1}^{d}\left[x_{i,k}-\frac{h_k}{2},\,x_{i,k}+\frac{h_k}{2}\right],
		\qquad |B_i|=\prod_{k=1}^{d}h_k,
		\]
        where \(x_{i,k}\) denotes the \(k\)-th coordinate of the center of the \(i\)-th reference bin, and \(h_k\) is the bin size along the \(k\)-th coordinate dimension, for \(k=1,\ldots,d\).
		In addition, the histogram density \(\widehat m_i\) approximates the average density over \(B_i\),
        \[
        \bar m_i=\frac{1}{|B_i|}\int_{B_i}m(x)\,dx.
        \]
        We further discretize the  integrals in the physics losses using a set of physics training points.  Firstly, let \(\mathcal P=\{z_j\}_{j=1}^{N_p}\subset\Omega\) denote the set of physics training points.
		
  For the low-dimensional examples, the integrals are approximated using Sobol quasi-Monte Carlo points, which are
		\begin{equation}
		\int_\Omega f(x)\,\mathrm dx
		\approx
		\frac{|\Omega|}{N_p}\sum_{j=1}^{N_p}f(z_j).
		\label{eq:sobol_quadrature}
		\end{equation}

For Case I introduced in Section~\ref{sec:formulation}, we define		\begin{equation}
		\bm r_j^{\rm eq}
		=
		D\nabla u_\theta(z_j)-\bm b (z_j)u_\theta(z_j),
		\qquad
		u_{\varepsilon,j}
		=
		\max\{u_\theta(z_j),\varepsilon_u\},
		\label{eq:eq_pointwise_residual}
		\end{equation}
        where $\varepsilon_{u}$ is given in (\ref{eq:weak_loss}).
Moreover, applying \eqref{eq:sobol_quadrature} to both the numerator and denominator of \eqref{eq:weak_loss} yields
		\begin{equation}
		\Lweak^{(N_p)}(\theta)
		=
		\frac{\displaystyle
		\sum_{j=1}^{N_p}
		\frac{\left|\bm r_j^{\rm eq}\right|^2}{u_{\varepsilon,j}}
		}{
		\displaystyle
		\sum_{j=1}^{N_p}u_\theta(z_j)
		}.
		\label{eq:weak_loss_discrete}
		\end{equation}

				For Case II, the residual at each physics point is defined using only the gradient component of the drift:
		\begin{equation}
		\bm r_j^{\rm orth}
		=
		D\nabla u_\theta(z_j)
		+u_\theta(z_j)\nabla V(z_j).
		\label{eq:orthogonal_pointwise_residual}
		\end{equation}
	
	Moreover, applying \eqref{eq:sobol_quadrature} to \eqref{eq:orthogonal_weak_loss} gives
		\begin{equation}
		\Lweak^{(N_p)}(\theta)
		=
		\frac{\displaystyle
		\sum_{j=1}^{N_p}
		\frac{\left|\bm r_j^{\rm orth}\right|^2}{u_{\varepsilon,j}}
		}{
		\displaystyle
		\sum_{j=1}^{N_p}u_\theta(z_j)
		}.
		\label{eq:orthogonal_weak_loss_discrete}
		\end{equation}

	For the high-dimensional example, uniform sampling over the computational domain becomes difficult as the dimension increases. We therefore use importance sampling for the physics points. Specifically, the physics points \(z_j\) are sampled from a proposal density \(q\), with the corresponding importance weights defined as

		\begin{equation}
		w_j=\frac{1}{|\Omega|q(z_j)},
		\label{eq:importance_weight_discrete}
		\end{equation}
		then the volume integral of an integrable function \(f\) is then approximated by
		\begin{equation}
		\int_\Omega f(x)\,\mathrm dx
		\approx
		\frac{|\Omega|}{N_p}
		\sum_{j=1}^{N_p}w_j f(z_j).
		\label{eq:importance_quadrature_discrete}
		\end{equation}
Accordingly, in the high dimension, the integral physics losses in two cases are given by
		\begin{equation}
		\Lweak^{(N_p)}(\theta)
		=
		\frac{\displaystyle
		\sum_{j=1}^{N_p}
		w_j\frac{\left|\bm r_j^{\rm eq/\rm {orth}}\right|^2}{u_{\varepsilon,j}}
		}{
		\displaystyle
		\sum_{j=1}^{N_p}
		w_j u_\theta(z_j)
		}.
		\label{eq:weak_loss_importance_discrete}
		\end{equation}

		\begin{remark}
		\label{subsec:scale_anchor}
        The physics losses chosen in two cases are invariant under positive rescaling of the density. Indeed, for any constant $c>0$,
        \[
        \Lweak(cu)=\Lweak(u),
        \]
        thanks to $D\nabla(cu)-\bm b(cu)=c(D\nabla u-\bm b u)$. This scale-invariant construction prevents the physics loss from being artificially reduced by simply decreasing the overall amplitude of $u$, thereby removing the dependence of the physics loss on the global scale of \(u\). We emphasize that this does not enforce the probability normalization $\int_\Omega u_\theta\,\mathrm dx=1$; the scale invariance applies only to the physics loss, while the scale of \(u_\theta\) is determined by the MC data pretraining.
\end{remark}

In Appendix~\ref{app:theory}, we derive error estimates
and prove convergence of the quadrature approximations
for the minimization problems with objective
$\mathcal L_{\rm W}(u)+\lambda_{\rm MC}\mathcal L_{\rm MC}(u)$
introduced in Section~\ref{sec:formulation}
for Cases~1 and~2.
        
		\section{Learning Methods}
		\label{sec:method}
With the mathematical framework established in Section \ref{sec:formulation}, we now present the three-stage learning procedure illustrated in Figure~\ref{fig:flowchart}. The following subsections detail the training strategy for each stage and the high-dimensional physics-sampling scheme.

		\begin{figure}[!htbp]
		\centering
		\resizebox{\textwidth}{!}{%
		\begin{tikzpicture}[
		 node distance=1.55cm,
		 box/.style={draw,rounded corners,align=center,minimum width=4.0cm,minimum height=1.25cm,fill=gray!8},
		 arrow/.style={-{Latex[length=3mm]},thick}
		]
		\node[box] (mc) {generate MC histogram};
		\node[box,right=of mc] (A) {Stage A\\MC data pretraining};
		\node[box,right=of A] (B) {Stage B\\data and physics Adam optimization};
		\node[box,right=of B] (C) {Stage C\\L-BFGS fine-tuning};
		\draw[arrow] (mc)--(A);
		\draw[arrow] (A)--(B);
		\draw[arrow] (B)--(C);
		\end{tikzpicture}%
		}
		\caption{Flowchart of the three-stage training procedure.}
		\label{fig:flowchart}
		\end{figure}
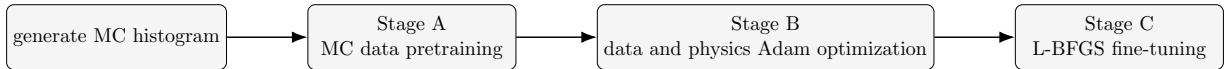
		
		\subsection{Neural Network Architecture and Monte Carlo Setup}
        Let $g_\theta(x)$ denote the raw network output. For  low-dimensional   problems, we use fully connected Tanh networks with
\[
u_\theta(x)=\operatorname{softplus}(g_\theta(x))+10^{-12}.\]
For high-dimensional problems, we instead use the log-density parameterization
\[
u_\theta(x)=
\exp\!\left[
\operatorname{clip}
\bigl(g_\theta(x),\ell_{\min},\ell_{\max}\bigr)
\right]+10^{-12},
\]
where $\ell_{\min}$ and $\ell_{\max}$ denote the prescribed lower and upper
clipping bounds for the network-predicted log-density, respectively.
The clipping is introduced solely for numerical stabilization, preventing
underflow or overflow in the exponential map when $g_\theta(x)$ takes
excessively small or large values. where the network learns an approximation to $\log u_\theta$, and the exponential map restores a positive density.

For the PINN baselines trained without MC data, the constraint \(\int_\Omega u_\theta\,\mathrm dx=1\) is imposed to fix the density scale, and the boundary loss is retained. For methods incorporating MC data, the density scale is determined by the data term in \eqref{eq:data_mse}.

        The raw MC samples are first accumulated into a regular histogram. To represent both high- and low-density regions, $N_{\rm ref}$ distinct reference bins are selected using a mixture of density-weighted and uniform sampling. Specifically, one subset is sampled according to the histogram density, while the other is sampled uniformly over all bins. The centers of the selected bins, together with their histogram densities, form the reference data set $\mathcal D_{\rm MC}
		=\{(x_i,\widehat m_i)\}_{i=1}^{N_{\rm ref}}$, and the fixed scale $s_{\rm{data}}$ for the data term is defined as
		\begin{equation}
		s_{\rm data}
		=\frac{1}{N_{\rm ref}}\sum_{i=1}^{N_{\rm ref}}\widehat m_i^2+\varepsilon_s,
		\qquad
		\widetilde{\Ldata}=\frac{\Ldata}{s_{\rm data}}.
	\label{eq:data_scale}
		\end{equation}
   Here, \(s_{\rm data}\) scales the MC data loss relative to the magnitude of the reference density, while \(\varepsilon_s>0\) prevents division by zero.
   
		We mention that since \(s_{\rm data}\) is fixed throughout training, this rescaling does not alter the minimizer of the data-fitting objective, but facilitates balancing the data and physics losses in Stages B-C.

The physics training pool $\mathcal P_{\rm train}$ is fixed throughout Stage B and sampled in mini-batches for gradient updates. An independent monitoring set $\mathcal P_{\rm mon}$, generated with a different random seed, is used for checkpoint selection and does not participate in training. For high-dimensional problems, importance sampling is used for the physics points as described in Subsection~\ref{subsec:8d_importance}.        
			
	In detail, after Stage B, the checkpoint with the lowest monitoring score is selected. The Adam-only method returns this checkpoint, whereas the Adam+L-BFGS method uses it to initialize Stage C. Stage C reuses the training and monitoring points from Stage B without resampling.  The MC data sampling and density construction procedure is summarized in Algorithm~\ref{alg:mc_reference}.

		\begin{algorithm}[H]
		\caption{Monte Carlo reference density construction}
		\label{alg:mc_reference}
		\footnotesize
		\begin{algorithmic}[1]
		\Require Drift $b$, diffusion $D$, domain $\Omega$, path number $N_{\rm path}$, time step $\Delta t$, burn-in time $T_{\rm burn}$, sampling time $T_{\rm sample}$, saving interval $\Delta t_{\rm save}$, histogram grid $\mathcal G_{\rm hist}$, reference count $N_{\rm ref}$, density fraction $\rho$, seed
		\Ensure Monte Carlo data set $\mathcal D_{\rm MC}$ and data scale $s_{\rm data}$
		\State Draw $N_{\rm path}$ initial particles and evolve them by Euler--Maruyama through the burn-in and sampling intervals
		\State After burn-in, save particle positions every $\Delta t_{\rm save}$ to form the stationary sample set $\mathcal S$
		\State Bin $\mathcal S$ on $\mathcal G_{\rm hist}$, convert counts to densities $\widehat m_j$, and normalize the discrete mass
		\State Set $N_\rho=\operatorname{round}(\rho N_{\rm ref})$ and $N_u=N_{\rm ref}-N_\rho$
		\State Sample $N_\rho$ boxes without replacement with probabilities proportional to $\widehat m_j+\varepsilon_{\rm ref}$, where $\varepsilon_{\rm ref}>0$ is a small numerical safeguard
		\State Sample $N_u$ boxes uniformly from the histogram grid
		\State Merge the two selections and, if necessary, add unsampled boxes until $N_{\rm ref}$ unique reference boxes are obtained
		\State Form $\mathcal D_{\rm MC}=\{(x_i,\widehat m_i)\}_{i=1}^{N_{\rm ref}}$ from the selected box centers and densities
		\State Set $s_{\rm data}=N_{\rm ref}^{-1}\sum_i\widehat m_i^2+\varepsilon_s$
		\State \Return $\mathcal D_{\rm MC}$ and $s_{\rm data}$
		\end{algorithmic}
		\end{algorithm}
		
		\subsection{Stage A: Monte Carlo Data Pretraining}

   Stage A provides a data-informed initialization of the density scale and dominant spatial structure before incorporating the Fokker-Planck constraint. The network is trained on the scaled data loss \(\widetilde{\Ldata}\) using AdamW with cosine-annealed learning rates and gradient clipping. A small weight decay of \(\lambda=10^{-9}\) is used. Training is performed over shuffled data mini-batches, and the parameters with the lowest data-fitting loss are retained as the Stage-A solution \(\theta_A\).  The detail of MC data pretraining is shown in Algorithm \ref{alg:stage_a}.

		\begin{algorithm}[H]
		\caption{MC data pretraining}
		\label{alg:stage_a}
		\footnotesize
		\begin{algorithmic}[1]
		\Require Initial parameters $\theta_0$, data set $\mathcal D_{\rm MC}$, data scale $s_{\rm data}$, epoch number $K_A$, batch size $B_d$, AdamW learning-rate schedule, clipping threshold $G_A$, check interval $K_{\rm check}^{A}$
		\Ensure Stage-A parameters $\theta_A$
		\State Set $\theta=\theta_0$; initialize AdamW and register the initial best checkpoint
		\For{$e=1,\ldots,K_A$}
		    \State Shuffle $\mathcal D_{\rm MC}$ and divide it into mini-batches
		    \For{each data batch $(X_b,Y_b)$}
		        \State Evaluate the scaled data loss $\widetilde{\mathcal L}_{\rm data}(\theta;X_b,Y_b)$
		        \State Backpropagate, clip the gradient norm to $G_A$, and take one AdamW step
		    \EndFor
		    \State Advance the prescribed learning-rate schedule
		    \If{$e$ is a check epoch or $e=K_A$}
		        \State Evaluate the full-data fitting error on $\mathcal D_{\rm MC}$ and save the checkpoint if it improves
		        \State Stop if the prescribed early-stopping rule is satisfied
		    \EndIf
		\EndFor
		\State Load the best Stage-A checkpoint
		\State \Return $\theta_A$
		\end{algorithmic}
		\end{algorithm}
		
		\subsection{Stage B: Joint Optimization: Variational Physics and MC Data Losses}

Stage B starts from $\theta_A$, with fixed and independent physics training and monitoring sets, $\mathcal P_{\rm train}$ and $\mathcal P_{\rm mon}$.  Sobol points are used for the low-dimensional  examples, whereas the high-dimensional  examples use the importance sampling described above. At $\theta_A$, we define the reference loss scales
\[
s_d=\max\{\widetilde{\Ldata}(\theta_A),\varepsilon_s\},
\qquad
s_p=\max\{\Lphys(\theta_A;\mathcal P_{\rm mon}),\varepsilon_s\}.
\]
Here, the fixed scales \(s_d\) and \(s_p\) normalize the data and physics losses relative to their values at the start of Stage B.

        The physics and data terms are optimized alternately using separate Adam optimizers with a prescribed update ratio. Physics points are cycled through the fixed training pool, while the MC data are reshuffled after each pass.
	
	We mention that checkpoint selection uses the monitoring score
		\begin{equation}
		S_B(\theta)
		=w_d\frac{\widetilde{\Ldata}(\theta)}{s_d}
		+w_p\frac{\Lphys(\theta;\mathcal P_{\rm mon})}{s_p},
\label{eq:stageB_score}
		\end{equation}
        where $s_d$ and $s_p$ scale the data and physics losses relative to their Stage-B initial values, and $w_d$ and $w_p$ specify their relative weights. The monitoring score is evaluated only at prescribed check times and does not participate in gradient updates. Since both components measure errors, a smaller $S_B$ indicates a better balance between data fitting and physics consistency.
		\begin{equation}
		k_B^\star=\arg\min_{k\in\mathcal K_B}S_B(\theta^{(k)}),
		\qquad
		\theta_B=\theta^{(k_B^\star)}.
		\label{eq:stageB_argmin}
		\end{equation}
 With fixed positive weights, a smaller \(S_B\) corresponds to a lower weighted data--physics error, the Stage B initial state is included in $\mathcal K_B$, ensuring that an updated state is selected only if it achieves a lower monitoring score.   See Algorithm \ref{alg:stage_b} for the details of the joint optimization.   
		\begin{algorithm}[H]
		\caption{Stage B: joint optimization: variational physics and MC data losses}
		\label{alg:stage_b}
		\footnotesize
		\begin{algorithmic}[1]
		\Require $\theta_A$, $\mathcal D_{\rm MC}$, $s_{\rm data}$, physics pool $\mathcal P_{\rm train}$, monitor set $\mathcal P_{\rm mon}$, physics loss $\mathcal L_{\rm phys}$, cycle number $K_B$, update counts $n_p,n_d$, batch sizes $B_p,B_d$, optimizer schedules, check interval $K_{\rm check}^{B}$, weights $w_p,w_d$
		\Ensure Stage-B parameters $\theta_B$ and the unchanged sets $\mathcal P_{\rm train},\mathcal P_{\rm mon}$
		\State Set $\theta=\theta_A$ and compute the initial scaled-data and monitored-physics losses $L_d^0,L_p^0$
		\State Fix $s_d=\max(L_d^0,\varepsilon_s)$ and $s_p=\max(L_p^0,\varepsilon_s)$ for the whole stage
		\State Initialize separate Adam optimizers for the data and physics objectives and register $\theta_A$ as the cycle-$0$ candidate
		\For{$k=1,\ldots,K_B$}
		    \For{$r=1,\ldots,n_p$}
		        \State Take a mini-batch from $\mathcal P_{\rm train}$ and perform one Adam update using $\mathcal L_{\rm phys}/s_p$
		        \State For the pure Strong PINN baseline, include the prescribed scaled boundary term
		    \EndFor
		    \For{$r=1,\ldots,n_d$}
		        \State Take a mini-batch from $\mathcal D_{\rm MC}$ and perform one Adam update using $\widetilde{\mathcal L}_{\rm data}$
		    \EndFor
		    \If{$k$ is a check cycle or $k=K_B$}
		        \State Evaluate $L_d^{\rm chk}$ on $\mathcal D_{\rm MC}$ and $L_p^{\rm chk}$ on the independent set $\mathcal P_{\rm mon}$
		        \State Compute $S_B=w_dL_d^{\rm chk}/s_d+w_pL_p^{\rm chk}/s_p$ and save the checkpoint if $S_B$ improves
		        \State Stop if the prescribed early-stopping rule is satisfied
		    \EndIf
		\EndFor
		\State Load the checkpoint with the smallest monitoring score
		\State \Return $\theta_B$, $\mathcal P_{\rm train}$, and $\mathcal P_{\rm mon}$
		\end{algorithmic}
		\end{algorithm}
		
		\subsection{Stage C: L-BFGS Fine-Tuning}
        Stage C is applied only to the ``L-BFGS” methods and starts from the Stage B solution \(\theta_B\), while Adam-only methods terminate at Stage B. The same physics training and monitoring sets are retained, so the two variants differ only by the additional L-BFGS fine-tuning.

    Let \(\mathcal P_C\subseteq\mathcal P_{\rm train}\) denote the physics points used in Stage C, with the full training pool used in the low-dimensional cases and a fixed subset used in the high-dimensional case. At the Stage-C initial state \(\theta_B\), we define
\(s_d^C=\max\{\widetilde{\Ldata}(\theta_B),\varepsilon_s\},
s_p^C=\max\{\Lphys(\theta_B;\mathcal P_C),\varepsilon_s\}.\)
L-BFGS then minimizes
\begin{align}\label{eq:lbfgs}
\mathcal L_{\rm LBFGS}(\theta)=
\alpha\frac{\widetilde{\Ldata}(\theta)}{s_d^C}
+
\beta\frac{\Lphys(\theta;\mathcal P_C)}{s_p^C}.
\end{align}

        The data and physics losses may differ substantially in numerical magnitude. Scaling them by \(s_d^C\) and \(s_p^C\) places the two terms on comparable relative scales, allowing \(\alpha\) and \(\beta\) to control their contributions to the Stage-C objective.  The scales \(s_d^C\) and \(s_p^C\) are fixed at \(\theta_B\) to balance the numerical magnitudes of the two losses.  Define

		\begin{equation}
		r_d^C(\theta)=\frac{\widetilde{\Ldata}(\theta)}{s_d^C},
		\qquad
		r_p^C(\theta)=\frac{\Lphys(\theta;\mathcal P_C)}{s_p^C}.
		\label{eq:stageC_dimensionless_ratio}
		\end{equation}
		Thus,  \eqref{eq:lbfgs} can be written as
		\begin{equation}
		\mathcal L_{\rm LBFGS}(\theta)
		=\alpha r_d^C(\theta)+\beta r_p^C(\theta).
		\label{eq:lbfgs_ratio_form}
		\end{equation}
        At \(\theta_B\), both ratios are approximately one and subsequently measure the relative change of each loss from its Stage C initial value.
	
        In summary, the scales \(s_d^C\) and \(s_p^C\) are fixed at \(\theta_B\) to balance the numerical magnitudes of the two losses. At \(\theta_B\), the corresponding scaled losses are approximately one and subsequently measure their changes relative to the Stage-C initial state. The weights \(\alpha\) and \(\beta\) then control the relative contributions of the data and physics terms.  The fine-tuning procedure is summarized in Algorithm~\ref{alg:stage_c}.

		\begin{algorithm}[H]
		\caption{Stage C:   L-BFGS fine-tuning}
		\label{alg:stage_c}
		\footnotesize
		\begin{algorithmic}[1]
		\Require $\theta_B$, $\mathcal D_{\rm MC}$, $\mathcal P_{\rm train}$, $\mathcal P_{\rm mon}$, physics loss $\mathcal L_{\rm phys}$, weights $\alpha,\beta$, iteration number $K_C$, check segment length $K_{\rm seg}$
		\Ensure Stage-C parameters $\theta_C$
		\State Set $\theta=\theta_B$ and choose the Stage-C data and physics sets only from the existing $\mathcal D_{\rm MC}$ and $\mathcal P_{\rm train}$
		\State Compute the fixed training scales $s_d^C,s_p^C$ and the corresponding monitoring scales at $\theta_B$
		\State Define the fixed objective $\mathcal L_{\rm LBFGS}=\alpha\,\widetilde{\mathcal L}_{\rm data}/s_d^C+\beta\,\mathcal L_{\rm phys}/s_p^C$
		\State Initialize L-BFGS with strong-Wolfe line search and register $\theta_B$ as the initial candidate
		\While{the L-BFGS iteration budget is not exhausted}
		    \State Run at most $K_{\rm seg}$ L-BFGS iterations on the fixed Stage-C objective
		    \State Evaluate the dimensionless data and physics ratios on the independent monitoring set
		    \State Form $S_C=\alpha r_d+\beta r_p$ and apply the data-protection guard in the high-dimensional case
		    \State Save the current checkpoint if $S_C$ improves and the guard is satisfied
		    \State Stop if the prescribed early-stopping rule is satisfied
		\EndWhile
		\State Load the accepted checkpoint with the smallest monitoring score
		\State \Return $\theta_C$
		\end{algorithmic}
		\end{algorithm}

        The L-BFGS iterations are periodically evaluated on the independent monitoring set using
\begin{equation}
S_C(\theta)
=\alpha\frac{\widetilde{\Ldata}_{,\rm mon}(\theta)}{s_{d,\rm mon}^C} +\beta\frac{\mathcal L_{\mathrm{phys,mon}}(\theta)}{s_{p,\rm mon}^C}. \label{eq:stageC_score} \end{equation}
The initial state is included as a candidate, and the parameters with the lowest monitoring score are retained. For the high-dimensional case, an additional data-error safeguard is imposed to prevent excessive degradation of the MC data fit. In particular, additional sampling treatments are introduced for the high-dimensional case; see Remark~\ref{subsec:8d_importance} for details.
		\begin{remark}
		\label{subsec:8d_importance}
		In high dimensions, uniform sampling places many physics points in low-density regions. We therefore use a mixture proposal of a local truncated Gaussian and a uniform distribution to improve sampling efficiency. Define \(q\) as the proposal density for sampling physics points, then we decompose
		\begin{equation}
		q(x)=a q_G(x)+(1-a)q_U(x),
		\qquad q_U(x)=|\Omega|^{-1},
		\end{equation}
        where \(q_G\) denotes the truncated Gaussian density, \(q_U\) is the uniform density over \(\Omega\), and \(a\in[0,1]\) is the mixture weight.  Specifically, its standard deviation along the \(i\)-th coordinate is set as
\[
\sigma_{q,i}=c_q\sqrt{\frac{D}{\lambda_i}},
\]
where $\lambda_i$ characterizes the drift strength near the linear equilibrium point and $c_q$ controls the proposal width. In the high-dimensional implementation, we set $a=0.70$ and $c_q=1.20$.  

We emphasize that since the physics points are sampled from the nonuniform proposal density \(q\), importance weights are introduced to correct the spatial integrals. For \(z_j\sim q\), define
\(w_j=\frac{1}{|\Omega|q(z_j)}.\)
Then, for an integrable function \(f\),
\[	
\int_\Omega f(x)\,\mathrm dx
\approx
\frac{|\Omega|}{N_p}
\sum_{j=1}^{N_p}w_jf(z_j).
\]
The importance correction is applied to all spatial integrals in the physics losses, reducing the variance associated with finite sampling. The detailed procedure is given in the following algorithm.

	\begin{algorithm}[H]
		\caption{High-dimensional importance sampling for physics points}
		\label{alg:importance}
		\footnotesize
		\begin{algorithmic}[1]
		\Require Point count $N$, box domain $\Omega=\prod_{i=1}^{d}[l_i,h_i]$, diffusion $D$, linear scales $\{\lambda_i\}_{i=1}^{d}$, Gaussian fraction $a$, scale factor $c_q$, seed
		\Ensure Weighted physics points $\{(x_j,w_j)\}_{j=1}^{N}$
		\State Set $\sigma_{q,i}=c_q\sqrt{D/\lambda_i}$ and split $N$ into truncated-Gaussian and uniform proposal counts
		\State Generate the truncated-Gaussian component by mapping scrambled Sobol points through the coordinatewise truncated-Gaussian inverse CDF
		\State Generate the uniform component by scrambled Sobol sampling over $\Omega$ and merge the two point sets
		\For{each generated point $x_j$}
		    \State Evaluate $q_G(x_j)$ and the mixture density $q(x_j)=a q_G(x_j)+(1-a)/|\Omega|$
		    \State Set the integration weight $w_j=1/(|\Omega|q(x_j))$
		\EndFor
		\State Apply a deterministic shuffle so that mini-batches mix the two proposal components
		\State \Return $\{(x_j,w_j)\}_{j=1}^{N}$
		\end{algorithmic}
		\end{algorithm}
		
		\FloatBarrier

\end{remark}

		\section{Numerical Experiments}
		\label{sec:numerics}
	In this section, we present numerical examples illustrating the implementation of the proposed algorithm.
    All experiments use the mean absolute percentage error (MAPE) over the effective-density region as the primary evaluation metric. Let \(m_\star\) denote an independent reference density. To exclude the near-zero tails, we set the threshold \(\tau=0.01\) and define
		\begin{equation}
		\mathcal I_\tau
		=\left\{i:m_\star(x_i)>\tau\max_j m_\star(x_j)\right\},
		\qquad \tau=0.01,
		\end{equation}
		\begin{equation}
		\mathrm{MAPE}_{\rm eff}
		=\frac{100\%}{|\mathcal I_\tau|}
		\sum_{i\in\mathcal I_\tau}
		\frac{|u_\theta(x_i)-m_\star(x_i)|}{m_\star(x_i)+10^{-14}}.
		\label{eq:mape}
		\end{equation}
        This metric measures the relative density error while avoiding large percentage errors in near-zero-density regions.

        We conduct an ablation study using seven comparison methods: Data-only NN, Strong PINN, Weak/Integral PINN, Strong+MC (Adam), Weak/Integral+MC (Adam), Strong+MC+L-BFGS, and Weak/Integral+MC+L-BFGS. The analytical stationary density is used only for final evaluation.  Detailed results are discussed for each example, with the complete MAPE comparison across all six experiments summarized in Section~\ref{subsec:cross_summary}.
		
		\subsection{Double-Well Potential in Two Dimensions}
		
		\paragraph{Model and Experimental Setup.}
		Consider
		\begin{equation}
		\bm b(x,y)=\begin{pmatrix}\alpha x-\beta x^3\\-\gamma y\end{pmatrix},
		\end{equation}
        the corresponding potential is
		\begin{equation}
		V(x,y)=\frac{\beta}{4}x^4-\frac{\alpha}{2}x^2+\frac{\gamma}{2}y^2.
		\end{equation}
        We set \(\alpha=\beta=\gamma=2\) and \(D=0.125\) on the computational domain \([-2,2]\times[-1.2,1.2]\). The MC reference is generated by Euler--Maruyama simulation with \(25{,}000\) trajectories and \(\Delta t=10^{-3}\). After a burn-in time of \(8\), states are sampled every \(0.1\) over the subsequent \(2\) time units, yielding \(500{,}000\) stationary samples. A histogram with bin width \(0.04\) is constructed, from which \(2{,}000\) reference bins are selected using \(70\%\) density-weighted and \(30\%\) uniform sampling. The network consists of four Tanh hidden layers of width \(128\) with a softplus output. The resulting MC histogram has an effective-region MAPE of \(9.85\%\) against the analytical stationary density, indicating the finite-sample error in the MC reference.

        For data-assisted methods, Stage A uses AdamW for up to \(1{,}200\) epochs with an initial learning rate of \(10^{-3}\) and cosine annealing. Stage B uses \(800\) physics training points and \(200\) monitoring points, with a mini-batch size of \(256\). One physics update and one data update are performed per cycle for up to \(6{,}000\) cycles, with initial learning rates \(1.2\times10^{-4}\) and \(6\times10^{-5}\), respectively. For the L-BFGS variants, Stage C runs for up to \(800\) iterations.
		
		\begin{figure}[!htbp]
		\centering
		\safeincludegraphics[width=0.98\textwidth,height=0.43\textheight,keepaspectratio]{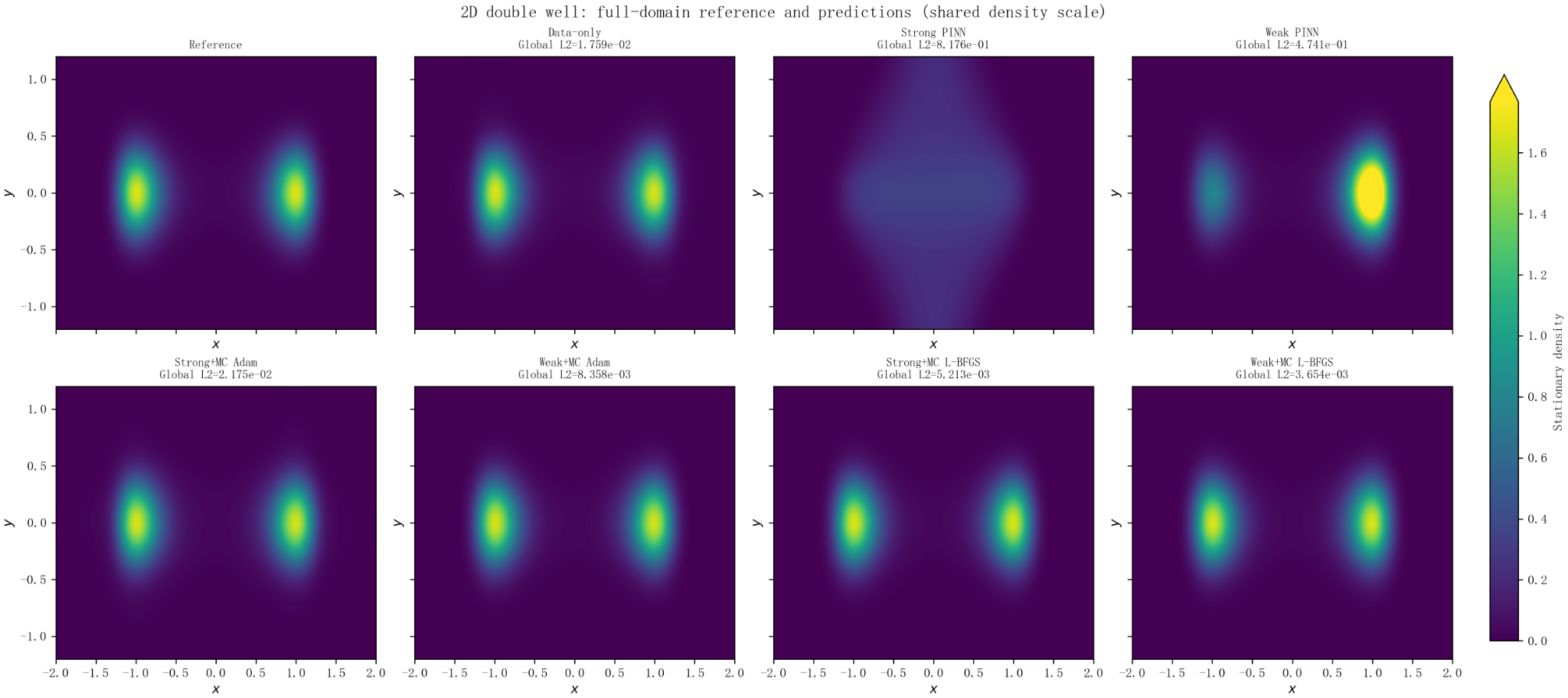}
		\caption{Two-dimensional double-well potential: analytical stationary density and full-domain predictions of the seven methods.}
		\label{fig:double-well-density}
		\end{figure}
		
		\paragraph{Results and Discussion}
		The reference density, the predictions of the seven methods, and the corresponding error distributions are shown in Figures.~\ref{fig:double-well-density}--\ref{fig:double-well-rel}; the method-level MAPE values are summarized in Figure~\ref{fig:double-well-mape}.

        Weak+MC+L-BFGS achieves the lowest MAPE of \(0.3452\%\). Along the three-stage procedure, the error decreases from \(7.2102\%\) for Data-only to \(1.2187\%\) after Stage B and further to \(0.3452\%\) after L-BFGS refinement, demonstrating the successive improvement from physics-informed training and Stage-C fine-tuning.

        The final Weak+MC+L-BFGS result also outperforms Strong+MC+L-BFGS (\(1.8557\%\)) and the physics-only PINN baselines. Moreover, its error is substantially lower than that of the original MC histogram (\(9.8451\%\)), indicating that incorporating the physics constraint improves upon the finite-sample MC reference.

		\begin{figure}[!htbp]
		\centering
		\safeincludegraphics[width=0.98\textwidth,height=0.43\textheight,keepaspectratio]{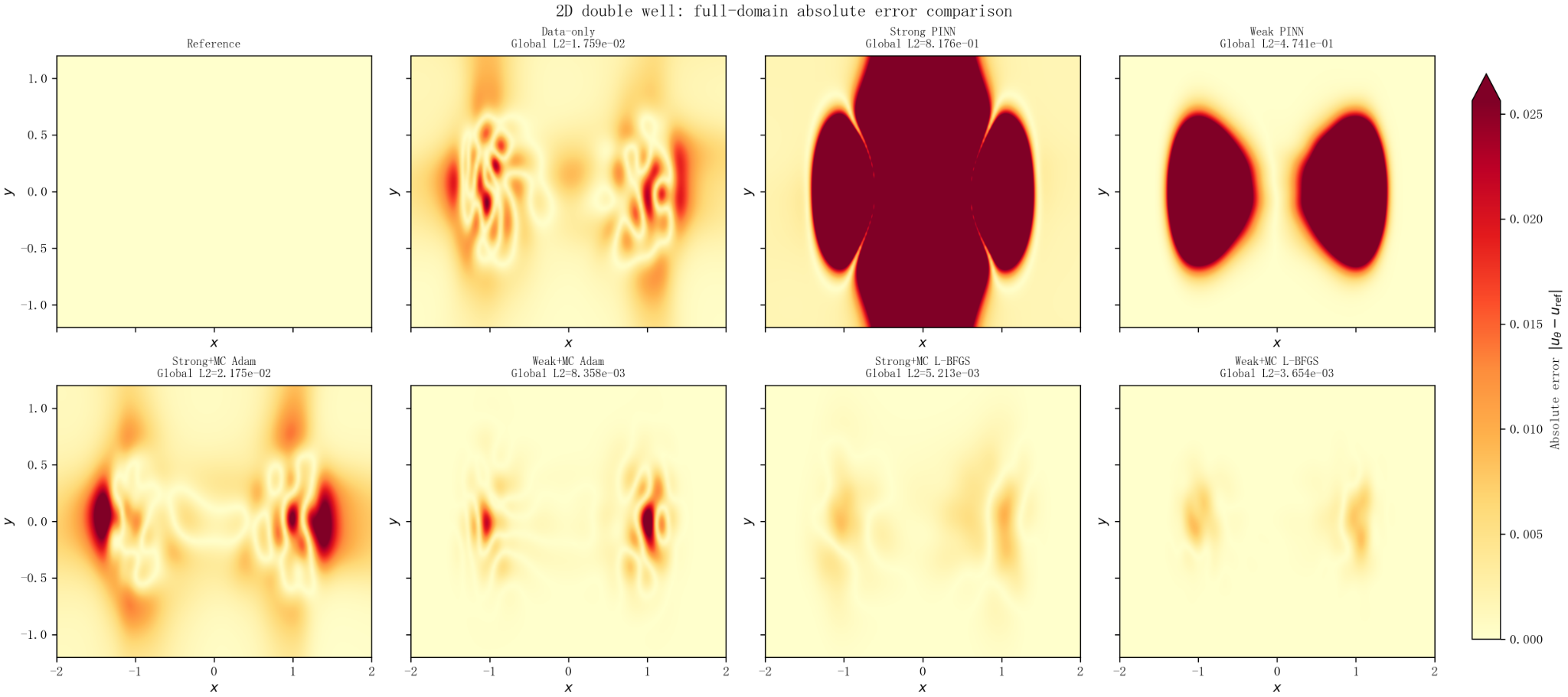}
		\caption{Two-dimensional double-well potential: full-domain absolute-error distributions of the seven methods, evaluated against the analytical reference density.}
		\label{fig:double-well-abs}
		\end{figure}
		
		\begin{figure}[!htbp]
		\centering
		\safeincludegraphics[width=0.98\textwidth,height=0.43\textheight,keepaspectratio]{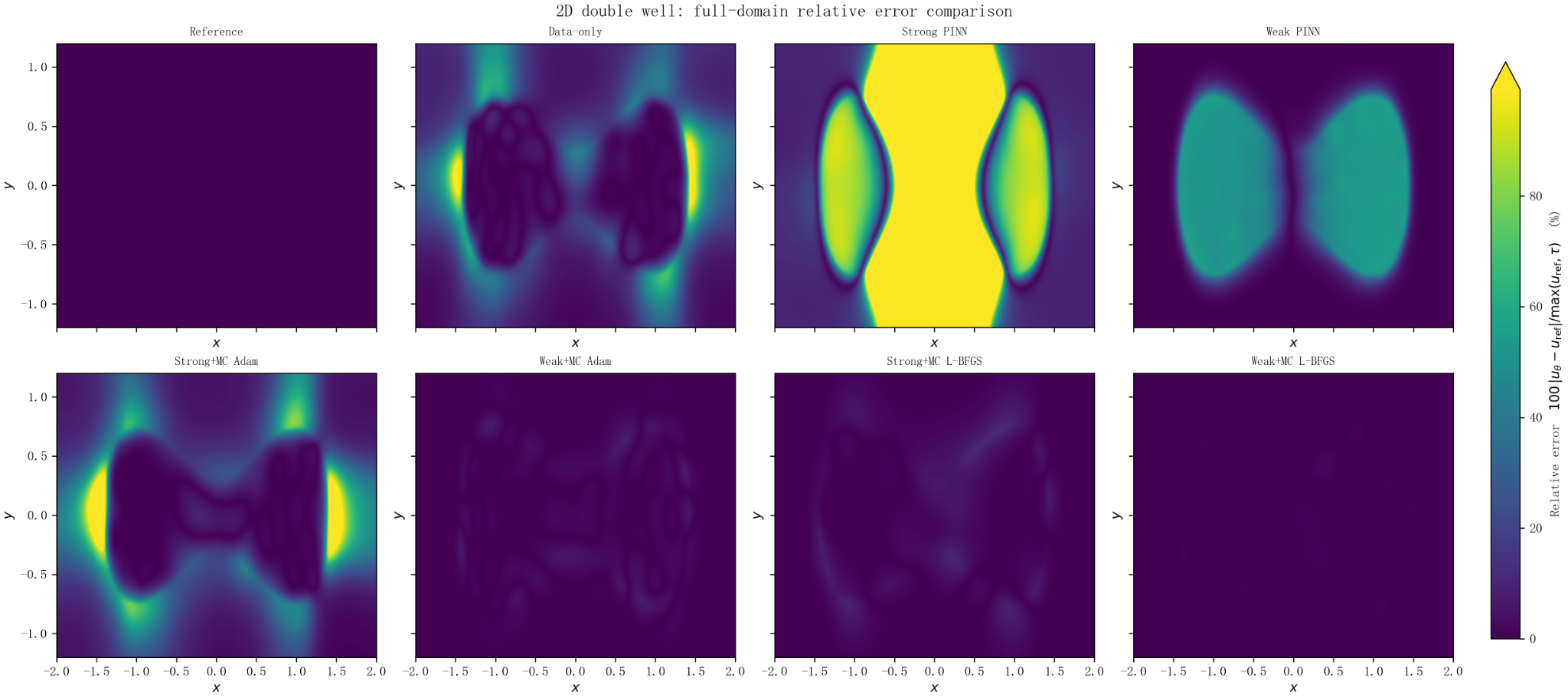}
		\caption{Two-dimensional double-well potential: full-domain pointwise relative-error distributions of the seven methods, evaluated against the analytical reference density.}
		\label{fig:double-well-rel}
		\end{figure}
		
		\begin{figure}[!htbp]
		\centering
		\safeincludegraphics[width=0.86\textwidth,keepaspectratio]{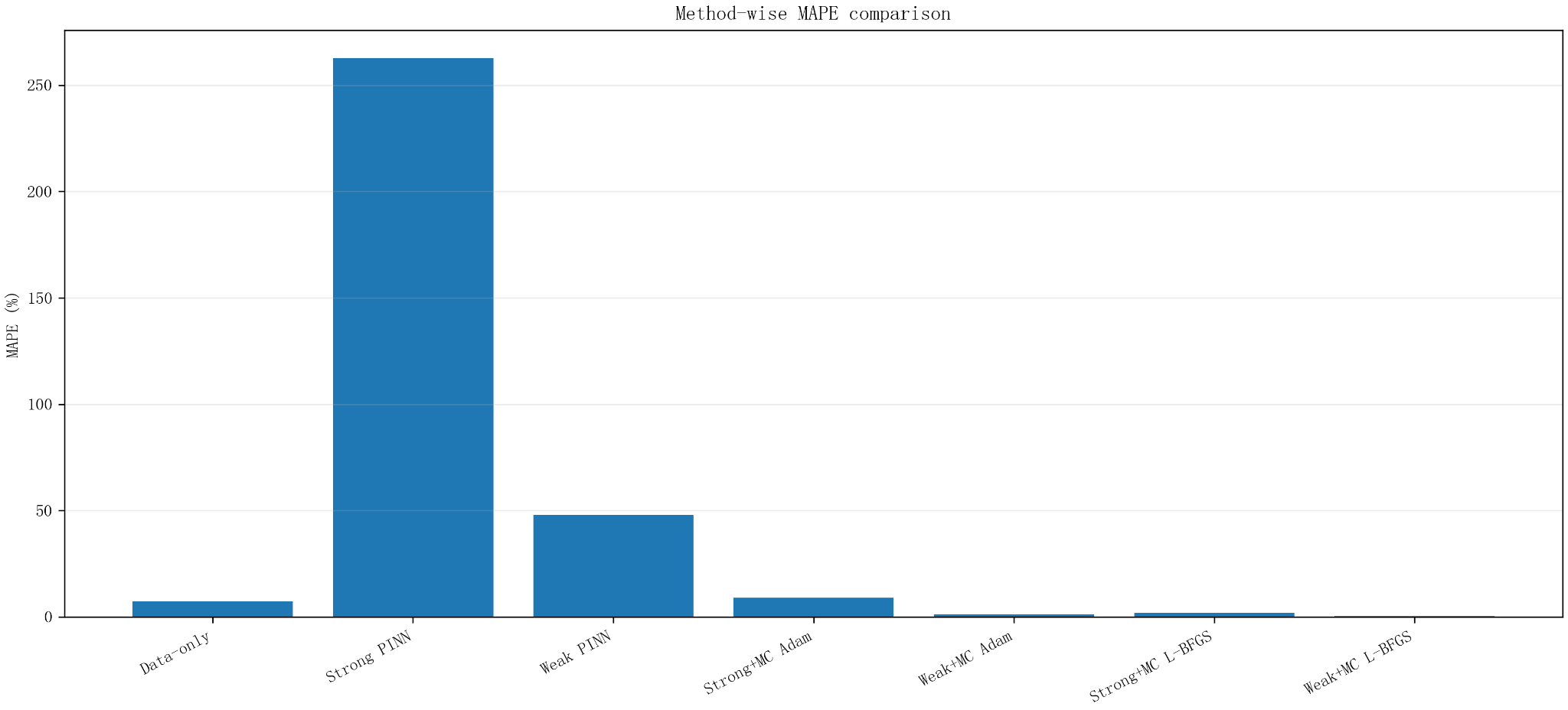}
		\caption{Two-dimensional double-well potential: effective-region MAPE comparison of the seven methods.}
		\label{fig:double-well-mape}
		\end{figure}

        \FloatBarrier
        \Needspace{8\baselineskip}
        \subsection{Rough Double-Well Potential with Orthogonal Rotation in Two Dimensions}
		\label{subsec:rough2d_rotation}
		
		\paragraph{Model and Experimental Setup.}
		Consider the rough potential
		\begin{equation}
		\psi(x,y)
		=
		\frac14(x^2-1)^2+\frac12y^2
		+\varepsilon^4
		\sin\left(\frac{2\pi x}{\varepsilon}\right)
		\sin\left(\frac{2\pi y}{\varepsilon}\right),
		\end{equation}
		where \(\varepsilon=0.1\) controls the oscillation scale. We define
		\begin{equation}
		\bm R
		=
		\nabla^\perp\psi
		=
		\begin{pmatrix}
		\partial_y\psi\\
		-\partial_x\psi
		\end{pmatrix},
		\qquad
		\bm b=-\nabla\psi+\bm R.
		\end{equation}
		By construction,
		\begin{equation}
		\bm R\cdot\nabla\psi=0,
		\qquad
		\nabla\cdot\bm R=0.
		\end{equation}
		With \(D=0.5\), the stationary density satisfies
        \[
        m_\star(x,y)
        \propto
        \exp\left(-\frac{\psi(x,y)}{D}\right).
        \]
        The corresponding probability flux satisfies
        \begin{equation}
        \J[m_\star]=\bm Rm_\star,
        \end{equation}
        and is generally nonzero. This example tests the pointwise-orthogonal extension of the variational formulation in a non-gradient stationary system.

		The computational domain is \([-2,2]^2\). The MC reference is generated using \(16{,}000\) Euler--Maruyama trajectories with time step \(\Delta t=0.002\). After a burn-in time of \(12\), states are sampled every \(0.05\) during the subsequent \(4\) time units, yielding \(1.28\times10^6\) stationary samples. A histogram with bin width \(0.05\) is constructed, and \(2{,}000\) reference bins are used for training. The network architecture is the same as in the preceding two-dimensional example. The MC histogram has an effective-region MAPE of \(9.6156\%\) against the analytical stationary density.
		
        \begin{figure}[!htbp]
		\centering
		\safeincludegraphics[width=0.98\textwidth,height=0.43\textheight,keepaspectratio]{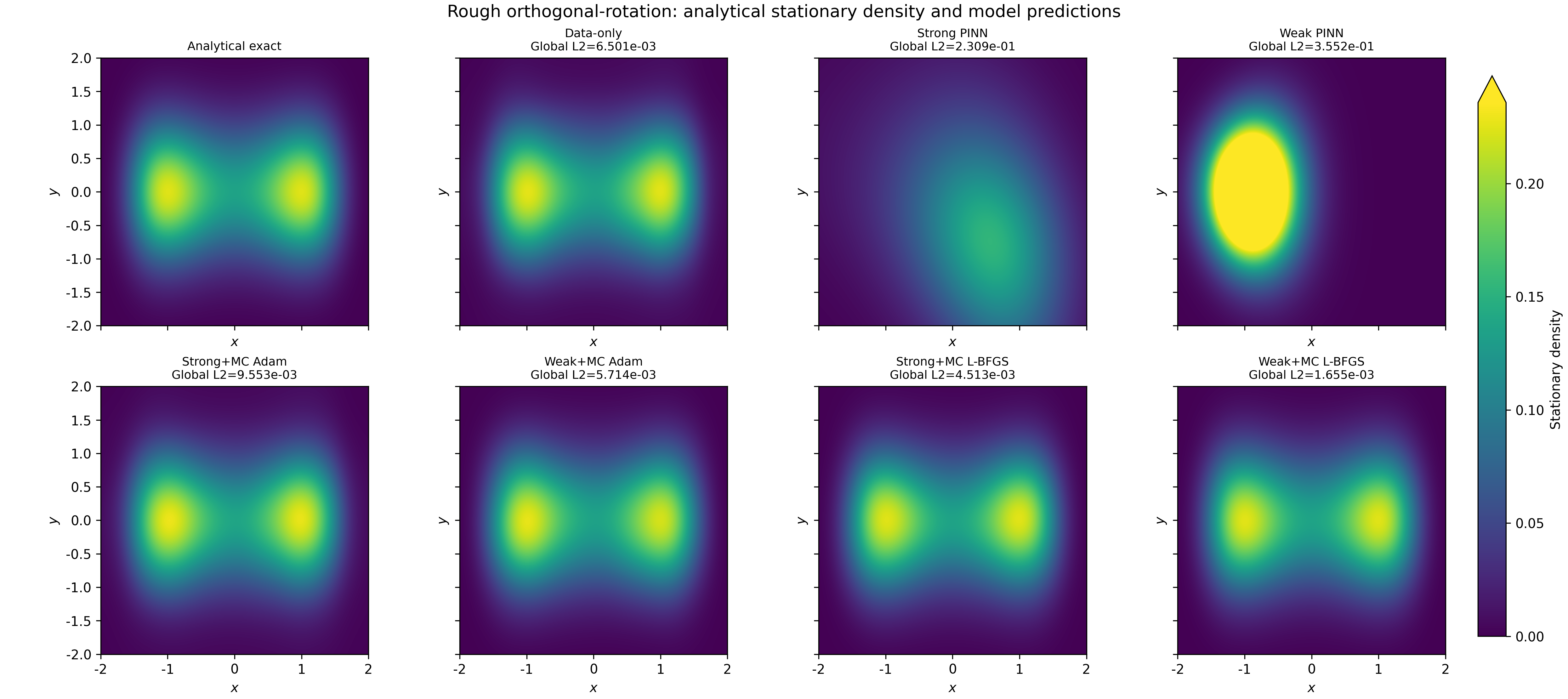}
		\caption{Two-dimensional rough double-well potential with orthogonal rotation: analytical stationary density and full-domain predictions of the seven methods.}
		\label{fig:rough2d-density}
		\end{figure}
		
		\paragraph{Results and Discussion}
		Weak+MC+L-BFGS achieves the lowest effective-region MAPE of \(0.4813\%\), compared with \(3.1981\%\) for Strong+MC+L-BFGS. Along the weak-form training procedure, the MAPE decreases from \(11.3059\%\) for Data-only to \(2.4990\%\) after Stage B and further to \(0.4813\%\) after L-BFGS refinement. The final value is also below the \(9.6156\%\) MAPE of the MC histogram.
		
		\begin{figure}[!htbp]
		\centering
		\safeincludegraphics[width=0.98\textwidth,height=0.43\textheight,keepaspectratio]{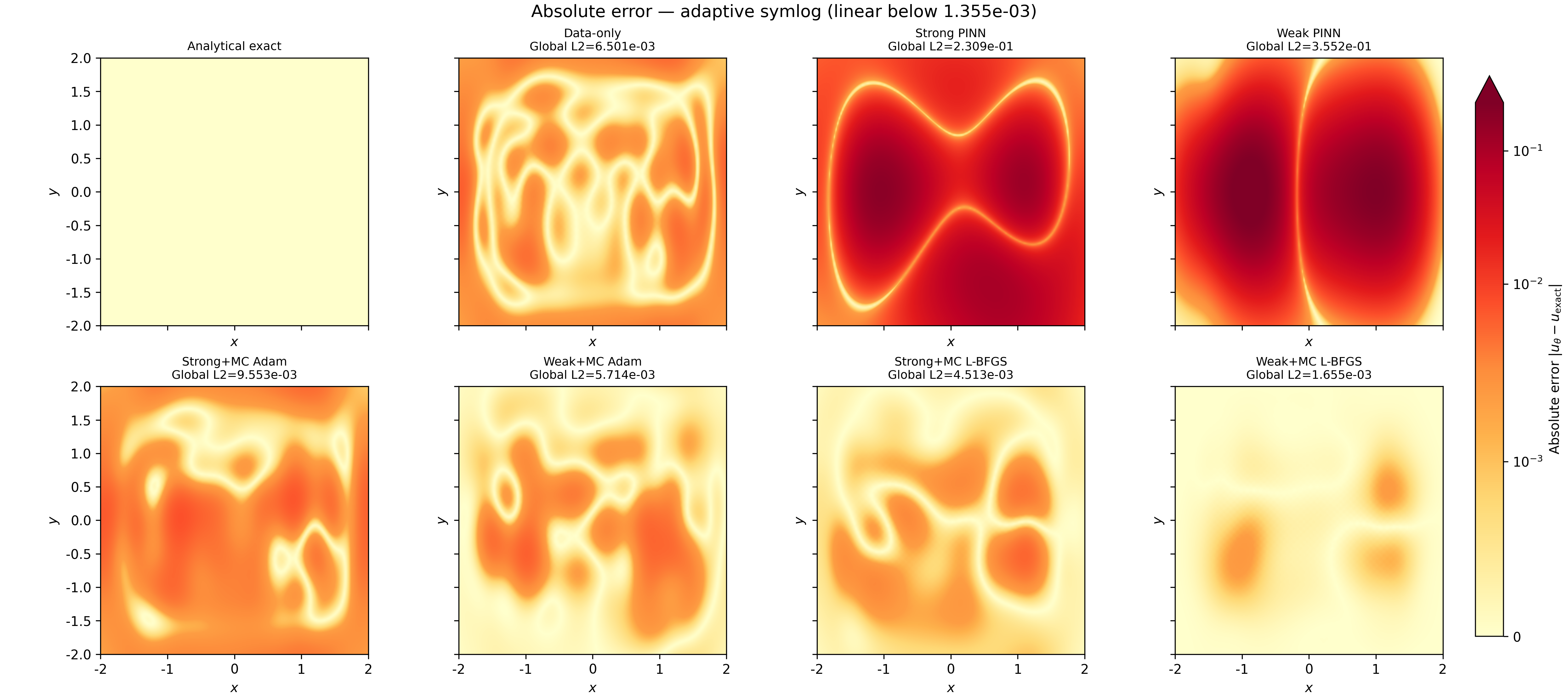}
		\caption{Two-dimensional rough double-well potential with orthogonal rotation: full-domain absolute-error distributions of the seven methods, evaluated against the analytical reference density.}
		\label{fig:rough2d-abs}
		\end{figure}
		
		\begin{figure}[!htbp]
		\centering
		\safeincludegraphics[width=0.98\textwidth,height=0.43\textheight,keepaspectratio]{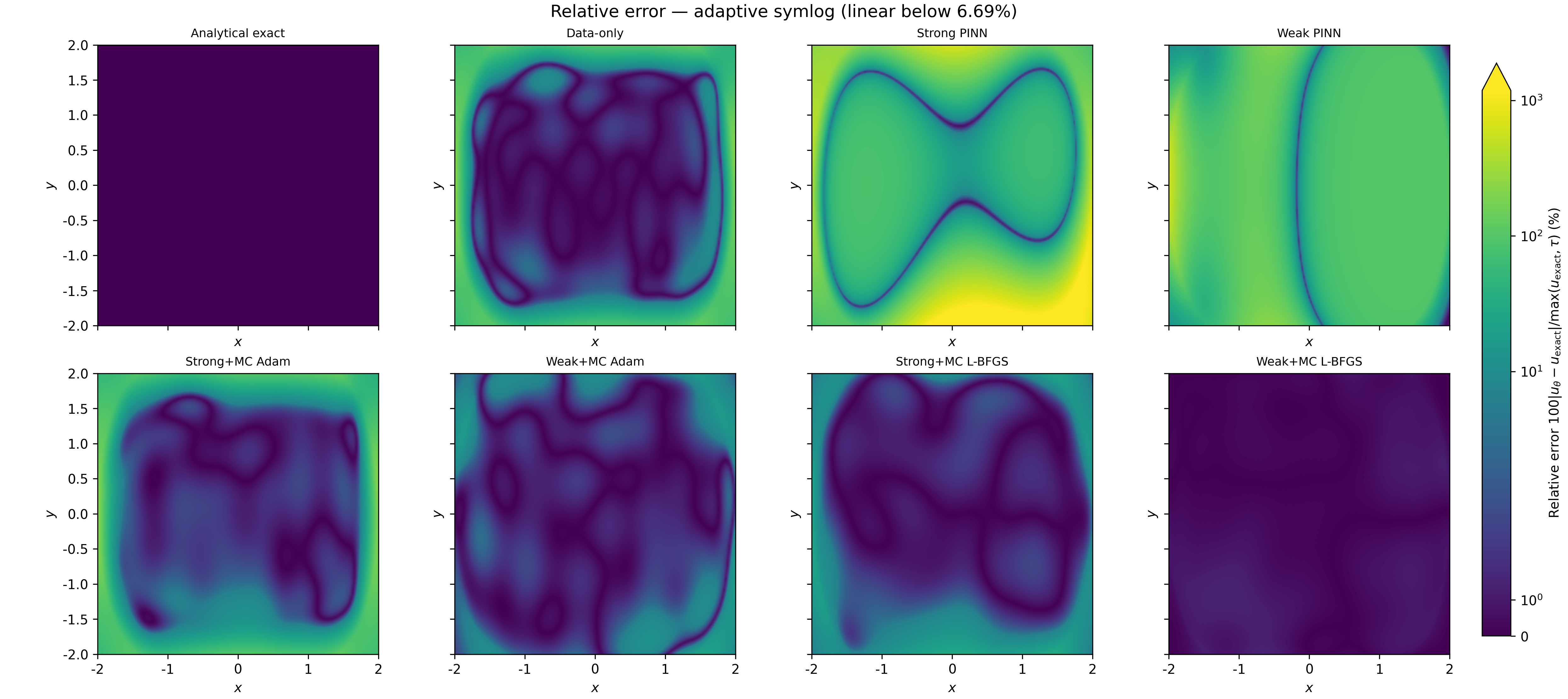}
		\caption{Two-dimensional rough double-well potential with orthogonal rotation: full-domain pointwise relative-error distributions of the seven methods, evaluated against the analytical reference density.}
		\label{fig:rough2d-rel}
		\end{figure}
		
		\begin{figure}[!htbp]
		\centering
		\safeincludegraphics[width=0.86\textwidth,keepaspectratio]{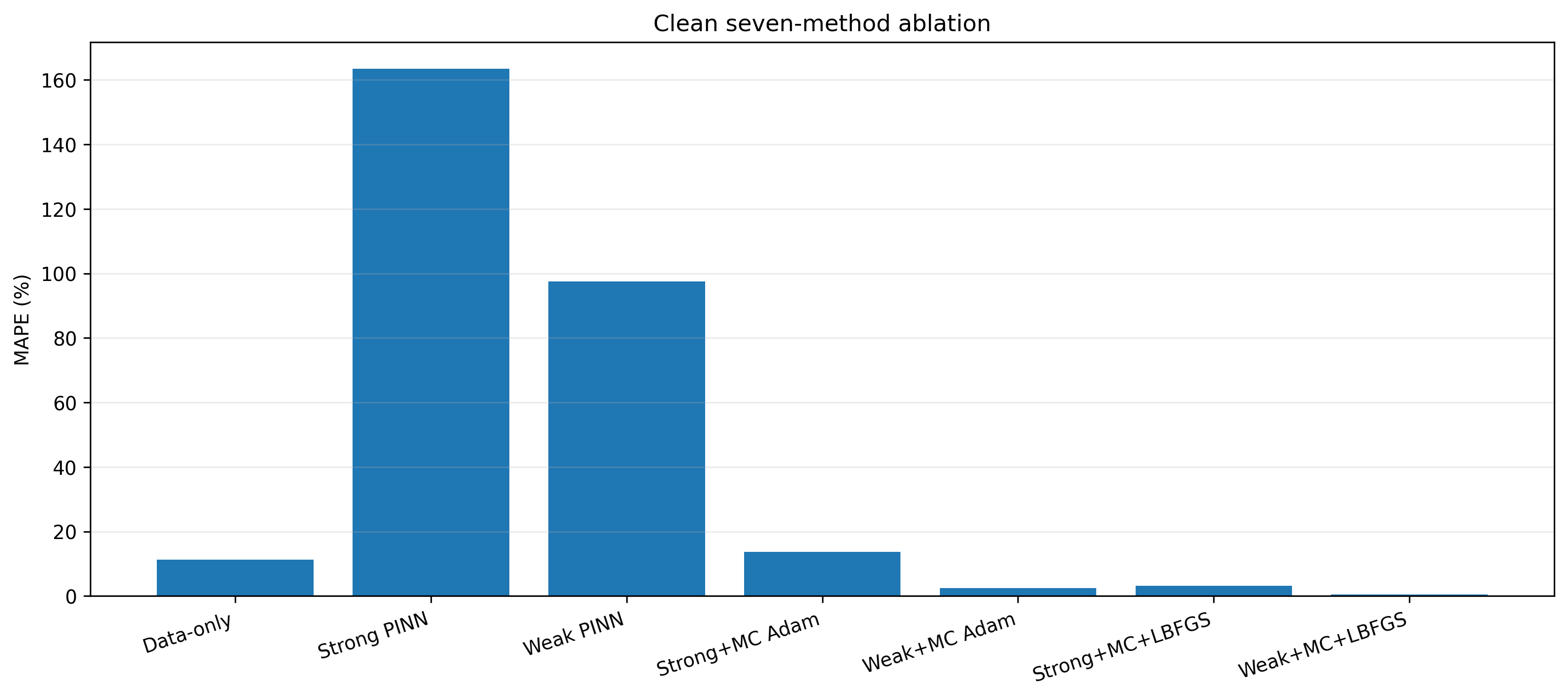}
		\caption{Two-dimensional rough double-well potential with orthogonal rotation: effective-region MAPE comparison of the seven methods.}
		\label{fig:rough2d-mape}
		\end{figure}

		\FloatBarrier
		\Needspace{8\baselineskip}
		\FloatBarrier
\Needspace{8\baselineskip}
\subsection{Four-Dimensional Nonequilibrium Steady State with a Test-Function Family}
\label{subsec:ness4d_test_family}

\paragraph{Model and Experimental Setup.}
Consider the four-dimensional rough potential
\begin{align}
 V(\bm x)
 =&\frac14(x_1^2-1)^2+\frac14(x_2^2-1)^2
 +\frac12x_3^2+\frac12x_4^2
 +\varepsilon^4\prod_{j=1}^{4}
 \sin\left(\frac{2\pi x_j}{\varepsilon}\right),
 \label{eq:ness4d_potential}
\end{align}
and let
\begin{equation}
 f(\bm x)=\frac{\gamma}{k\sqrt{2}}
 \left[\sin(kx_1)\sin(kx_3)+\sin(kx_2)\sin(kx_4)\right],
 \label{eq:ness4d_current_potential}
\end{equation}
where \(\gamma=0.06\) and \(k=8\).  With
\[
 Q=\begin{pmatrix}
 0&1&0&0\\-1&0&0&0\\0&0&0&1\\0&0&-1&0
 \end{pmatrix},
 \qquad Q^{\mathsf T}=-Q,
\]
define
\begin{equation}
 \bm R=Q\nabla f-\frac{f}{D}Q\nabla V,
 \qquad
 \bm b=-\nabla V+\bm R .
 \label{eq:ness4d_drift}
\end{equation}
For the Gibbs density
\(
 m_\star(\bm x)\propto\exp[-V(\bm x)/D]
\), one has
\[
 m_\star\bm R=Q\nabla(fm_\star),
 \qquad
 \nabla\cdot(m_\star\bm R)=0,
\]
and hence
\(\J[m_\star]=m_\star\bm R\) is generally nonzero.  This is a
nonequilibrium steady state and does not require
\(\bm R\cdot\nabla V=0\).

We set \(D=0.125\), \(\varepsilon=0.10\), and use the domain
\([-2,2]^2\times[-1.2,1.2]^2\).  The MC reference is generated with
\(210{,}000\) Euler--Maruyama trajectories, \(\Delta t=10^{-3}\), a burn-in
time of \(8\), and sampling every \(0.1\) over the subsequent \(2\) time
units.  This gives approximately \(4.2\times10^6\) stationary samples.  A
histogram with bin width \(0.2\) is formed and \(10{,}000\) reference bins are
used for training.  The network has four Tanh hidden layers of width \(160\)
and a softplus output.  The weak loss uses the \(K=330\) lowest physical
Dirichlet-Laplacian sine modes
\begin{equation}
 \varphi_{\bm n}(\bm x)
 =\prod_{j=1}^{4}
 \sin\left(\frac{n_j\pi(x_j-a_j)}{L_j}\right),
 \qquad
 \lambda_{\bm n}=\pi^2\sum_{j=1}^{4}\left(\frac{n_j}{L_j}\right)^2,
 \label{eq:ness4d_sine_family}
\end{equation}
with \(L_j=b_j-a_j\), exact \(H_0^1\) normalization, and equal weights
\(\omega_k=1\).  Stage A uses 1300 epochs, Stage B uses 3000 cycles with
three physics updates and one data update per cycle, and Stage C uses 600
L-BFGS iterations.

\begin{figure}[!htbp]
\centering
\safeincludegraphics[
    width=0.98\textwidth,
    height=0.43\textheight,
    keepaspectratio
]{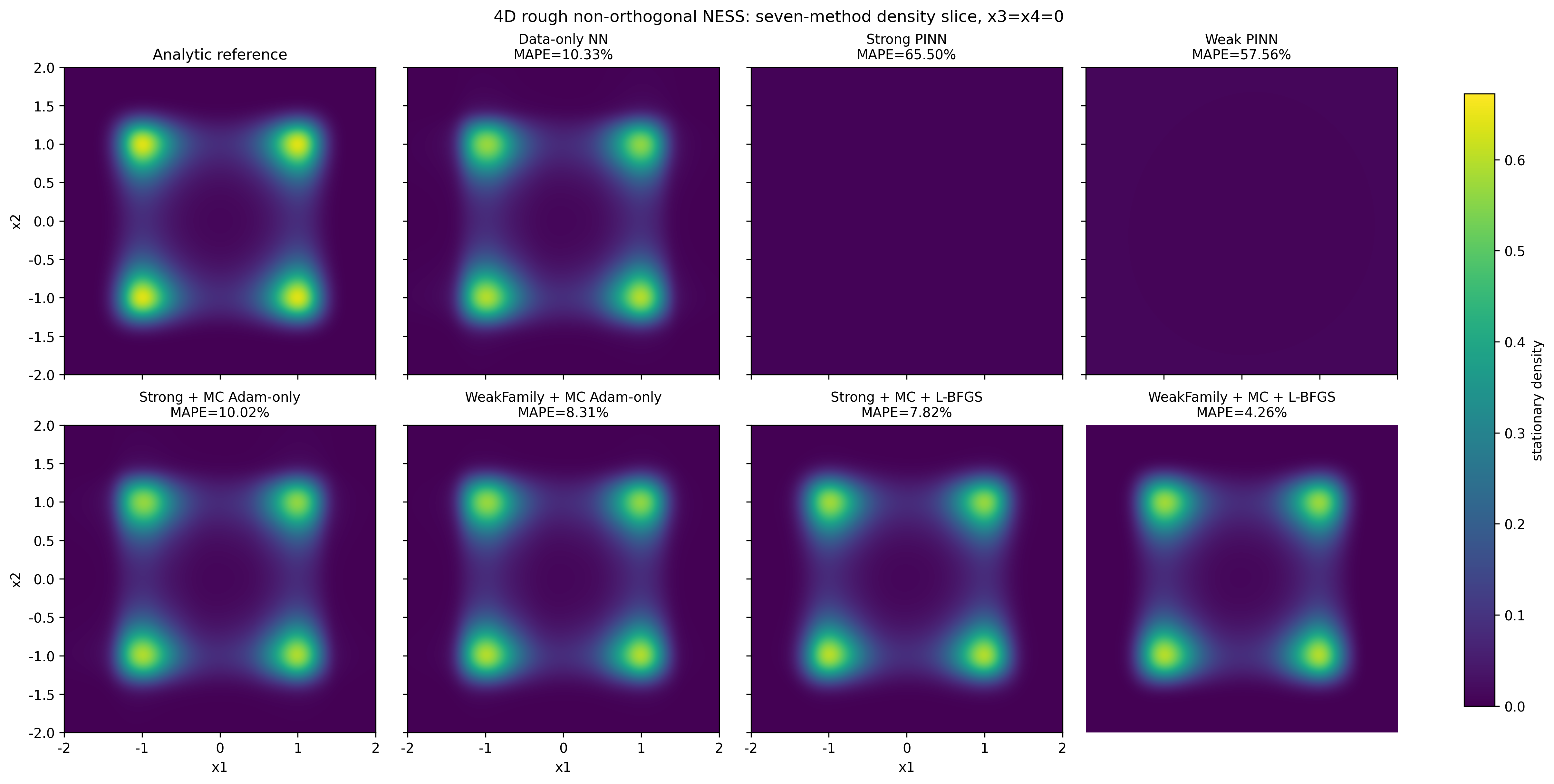}
\caption{Four-dimensional nonequilibrium steady state: reference density and
seven-method predictions on the slice $x_3=x_4=0$.}
\label{fig:ness4d-density}
\end{figure}

\paragraph{Results and Discussion.}
The test-function weak loss gives the lowest MAPE among the seven methods.
WeakFamily+MC+L-BFGS obtains an effective-region MAPE of \(4.2620\%\),
compared with \(7.8237\%\) for Strong+MC+L-BFGS and \(10.3307\%\) for the
Data-only NN.  The corresponding improvement over the strong-form hybrid is
\(3.5617\) percentage points.  The result shows that the stationary-current
balance can be imposed through a fixed test-function family while retaining
the MC data anchor and the three-stage training procedure.

\begin{figure}[!htbp]
\centering
\safeincludegraphics[
    width=0.98\textwidth,
    height=0.43\textheight,
    keepaspectratio
]{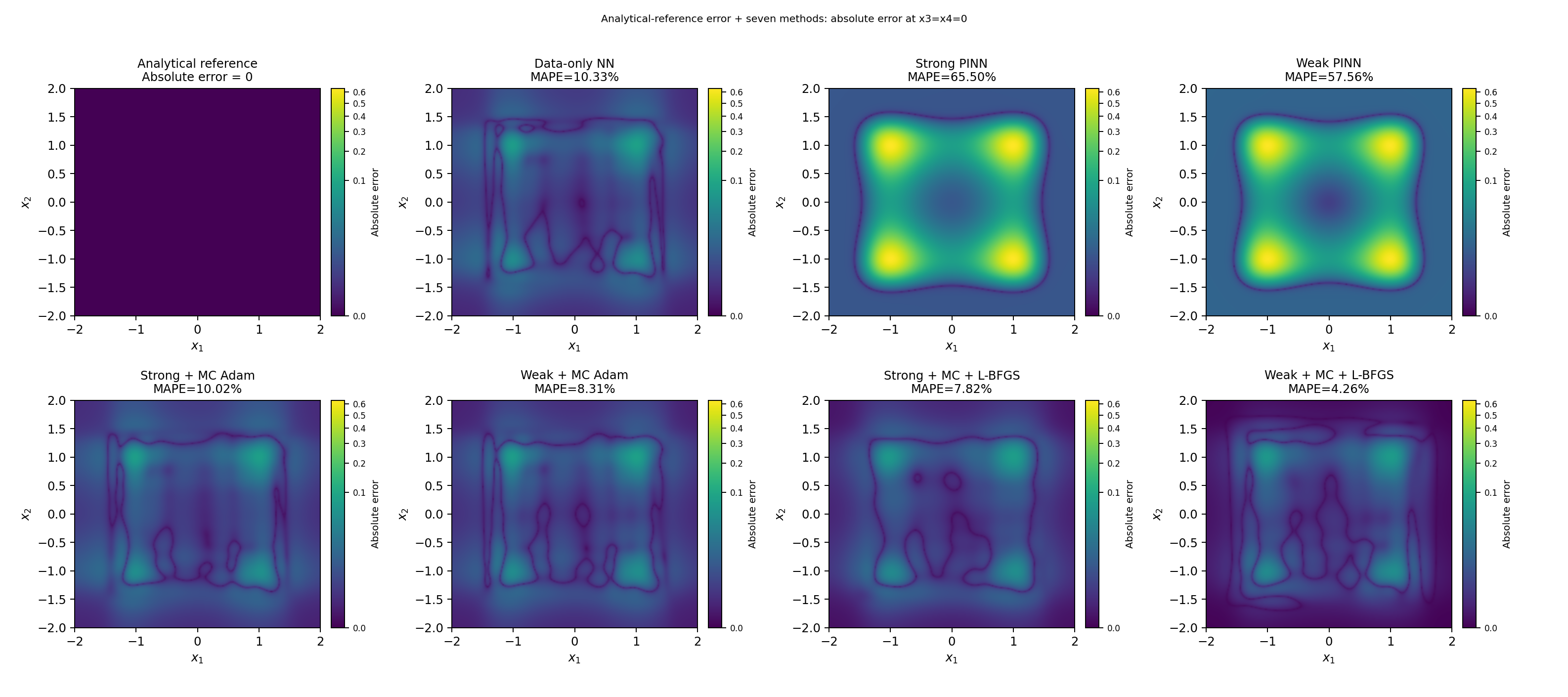}
\caption{Four-dimensional nonequilibrium steady state: absolute-error
distributions on the slice $x_3=x_4=0$.}
\label{fig:ness4d-abs}
\end{figure}

\begin{figure}[!htbp]
\centering
\safeincludegraphics[
    width=0.98\textwidth,
    height=0.43\textheight,
    keepaspectratio
]{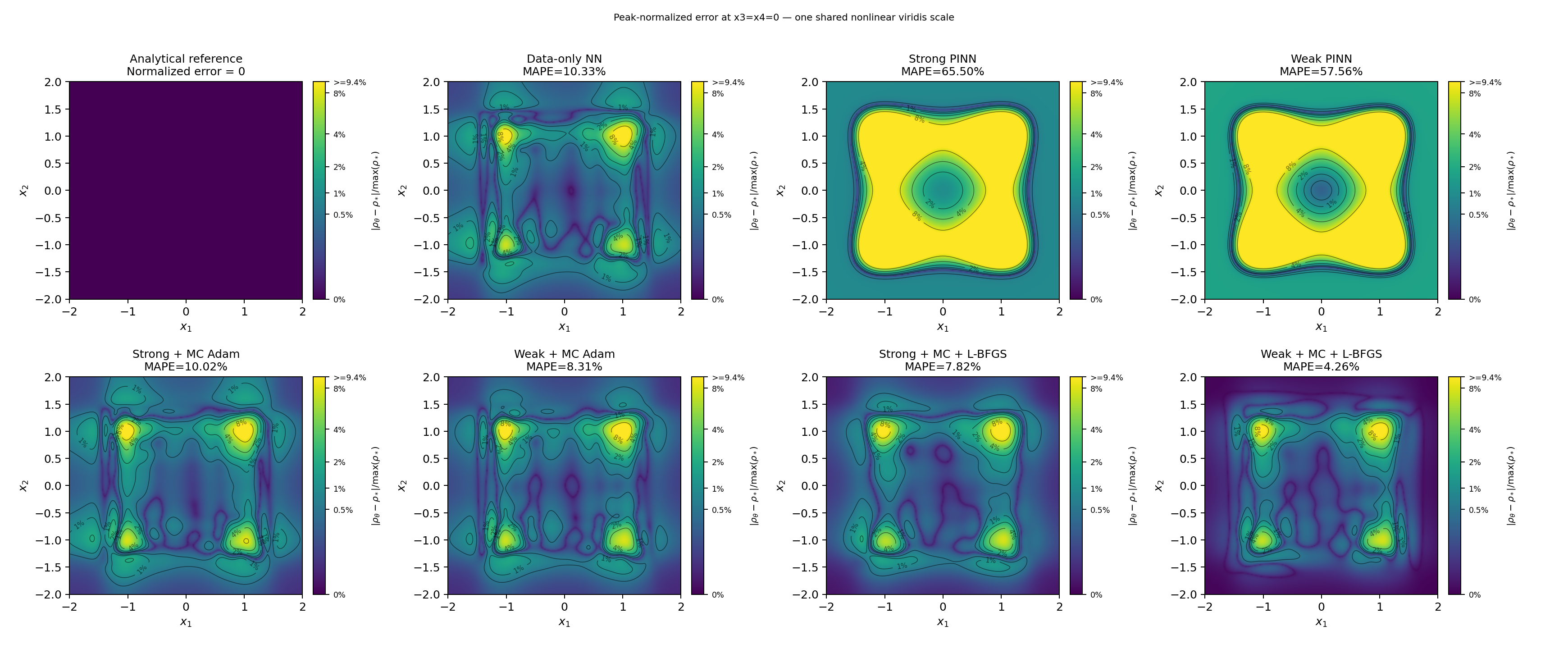}
\caption{Four-dimensional nonequilibrium steady state: pointwise
relative-error distributions on the slice $x_3=x_4=0$.}
\label{fig:ness4d-rel}
\end{figure}

\begin{figure}[H]
\centering
\safeincludegraphics[
    width=0.86\textwidth,
    keepaspectratio
]{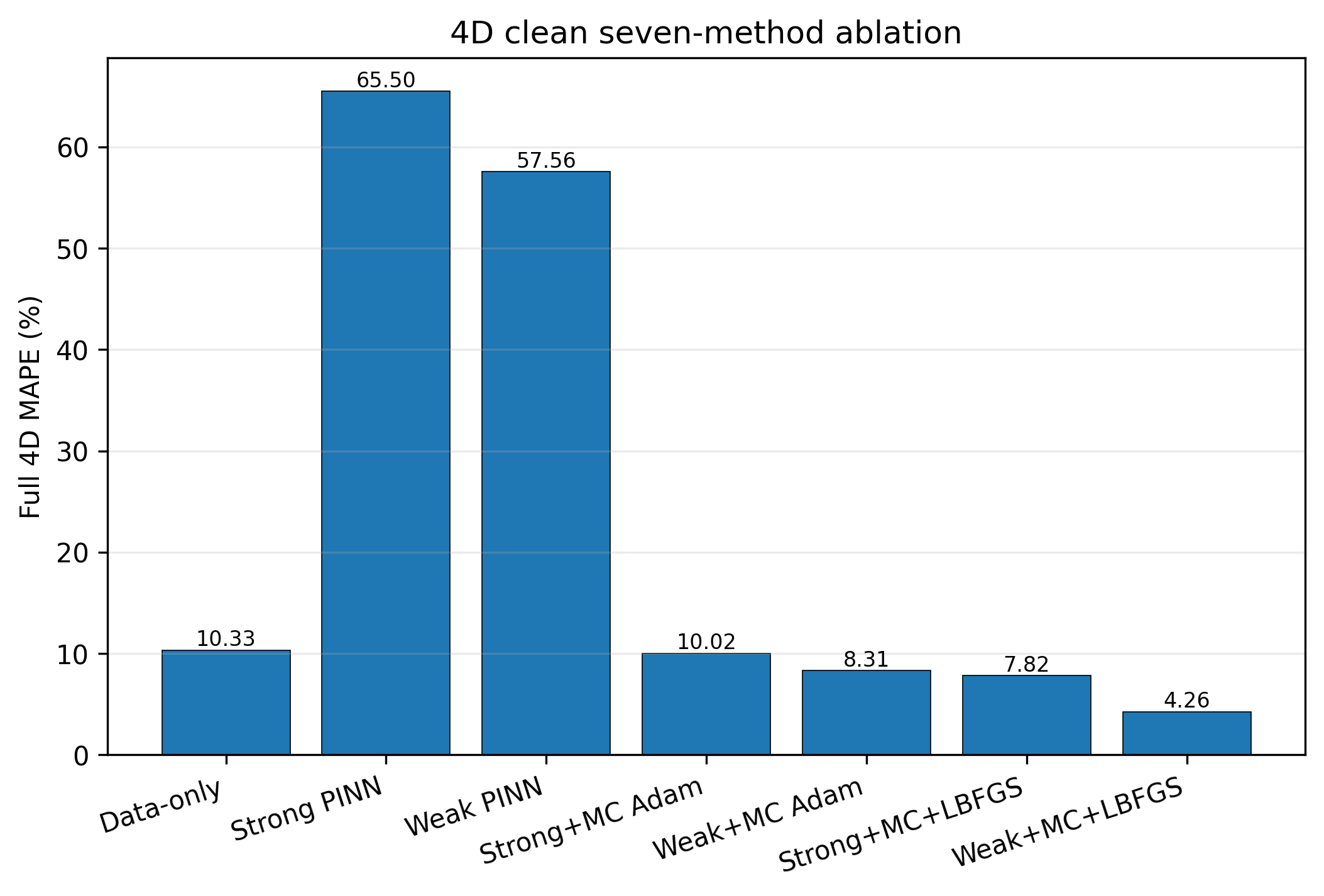}
\caption{Four-dimensional nonequilibrium steady state: effective-region MAPE
comparison of the seven methods.}
\label{fig:ness4d-mape}
\end{figure}

\FloatBarrier

\subsection{Rough Quadruple-Well Potential in Four Dimensions}
		
		\paragraph{Model and Experimental Setup.}
		Consider the four-dimensional potential
        \begin{align}
        V(\bm x)
        =&\frac14(x_1^2-1)^2+\frac14(x_2^2-1)^2
        +\frac12x_3^2+\frac12x_4^2\nonumber\\
        &+\varepsilon^4\prod_{j=1}^{4}
        \sin\left(\frac{2\pi x_j}{\varepsilon}\right).
        \end{align}
        Here \(\varepsilon=0.1\) controls the oscillation scale. The drift is \(\bm b=-\nabla V\), with \(D=0.125\), on the computational domain \([-2,2]^2\times[-1.2,1.2]^2\). The first two coordinates have double-well structure, while the remaining two are quadratically confined, so the stationary density is concentrated near \((\pm1,\pm1,0,0)\).

		The MC reference is generated using \(210000\) trajectories with time step \(\Delta t=10^{-3}\). After a burn-in time of \(8\), states are sampled every \(0.1\) during the subsequent \(2\) time units, yielding approximately \(4.2\times10^6\) stationary samples. A histogram with bin width \(0.2\) in each coordinate is constructed, and \(8000\) reference bins are used for training. The network uses four Tanh hidden layers of width \(128\) with a softplus output. The MC histogram has an effective-region MAPE of \(9.6676\%\).
		
		Stage B uses \(4000\) physics training points and \(1000\) monitoring points and runs for at most \(3000\) cycles. Each cycle contains two physics updates followed by one data update.

		\begin{figure}[!htbp]
		\centering
		\safeincludegraphics[width=0.98\textwidth,height=0.43\textheight,keepaspectratio]{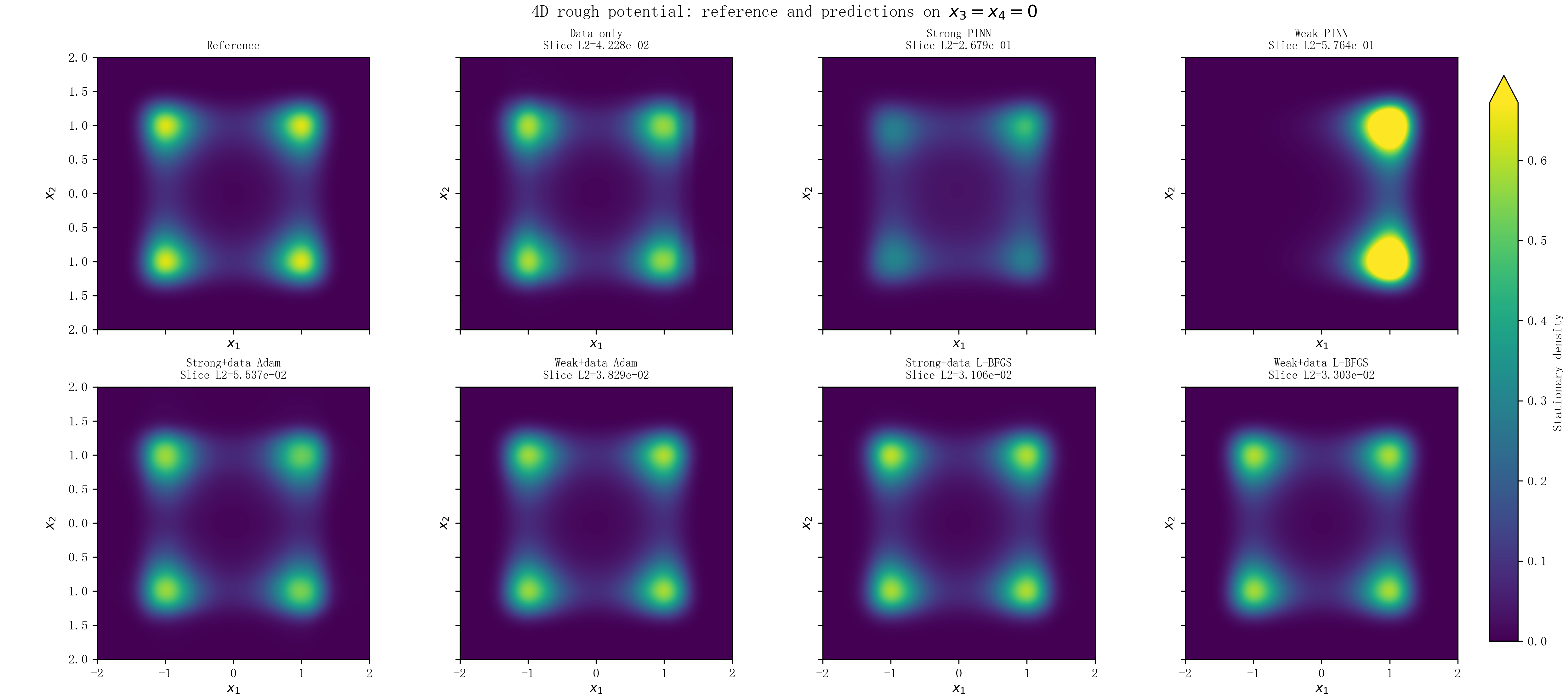}
		\caption{Four-dimensional rough-potential: reference stationary density and predictions of the seven methods on the slice $x_3=x_4=0$.}
		\label{fig:rough4d-density}
		\end{figure}
		
		\paragraph{Results and Discussion}
		On the \(x_3=x_4=0\) slice, the reference density, seven-method predictions, and corresponding errors are shown in Figures.~\ref{fig:rough4d-density}--\ref{fig:rough4d-rel}; the MAPE comparison is summarized in Figure~\ref{fig:rough4d-mape}.
		
		Weak+MC+L-BFGS achieves the lowest effective-region MAPE of \(1.7430\%\), compared with \(6.3159\%\) for Strong+MC+L-BFGS. Along the weak-form training procedure, the MAPE decreases from \(8.2217\%\) for Data-only to \(3.5404\%\) after Stage B and further to \(1.7430\%\) after L-BFGS fine-tuning. The final value is also below the \(9.6676\%\) MAPE of the MC histogram.

		\begin{figure}[!htbp]
		\centering
		\safeincludegraphics[width=0.98\textwidth,height=0.43\textheight,keepaspectratio]{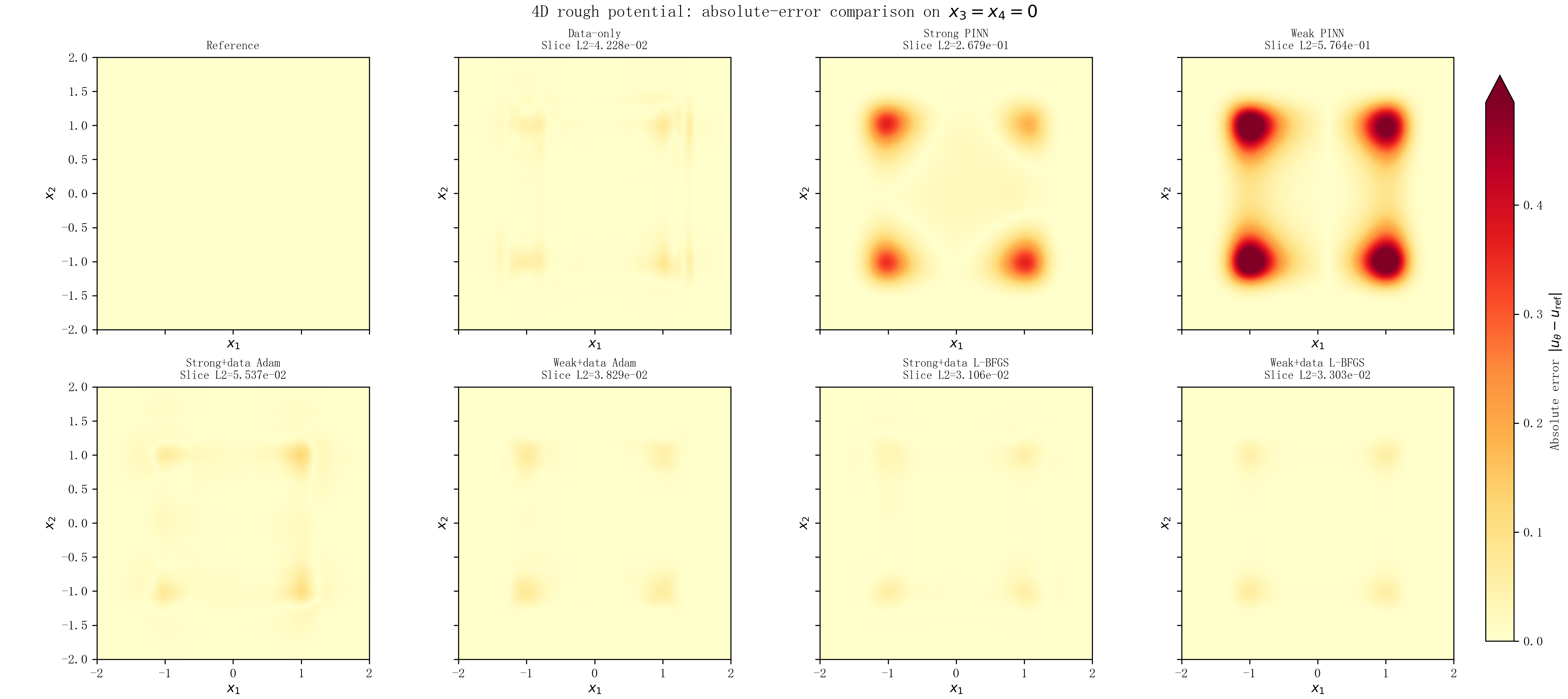}
		\caption{Four-dimensional rough-potential: absolute-error distributions of the seven methods on the slice \(x_3=x_4=0\).}
		\label{fig:rough4d-abs}
		\end{figure}
		
		\begin{figure}[!htbp]
		\centering
		\safeincludegraphics[width=0.98\textwidth,height=0.43\textheight,keepaspectratio]{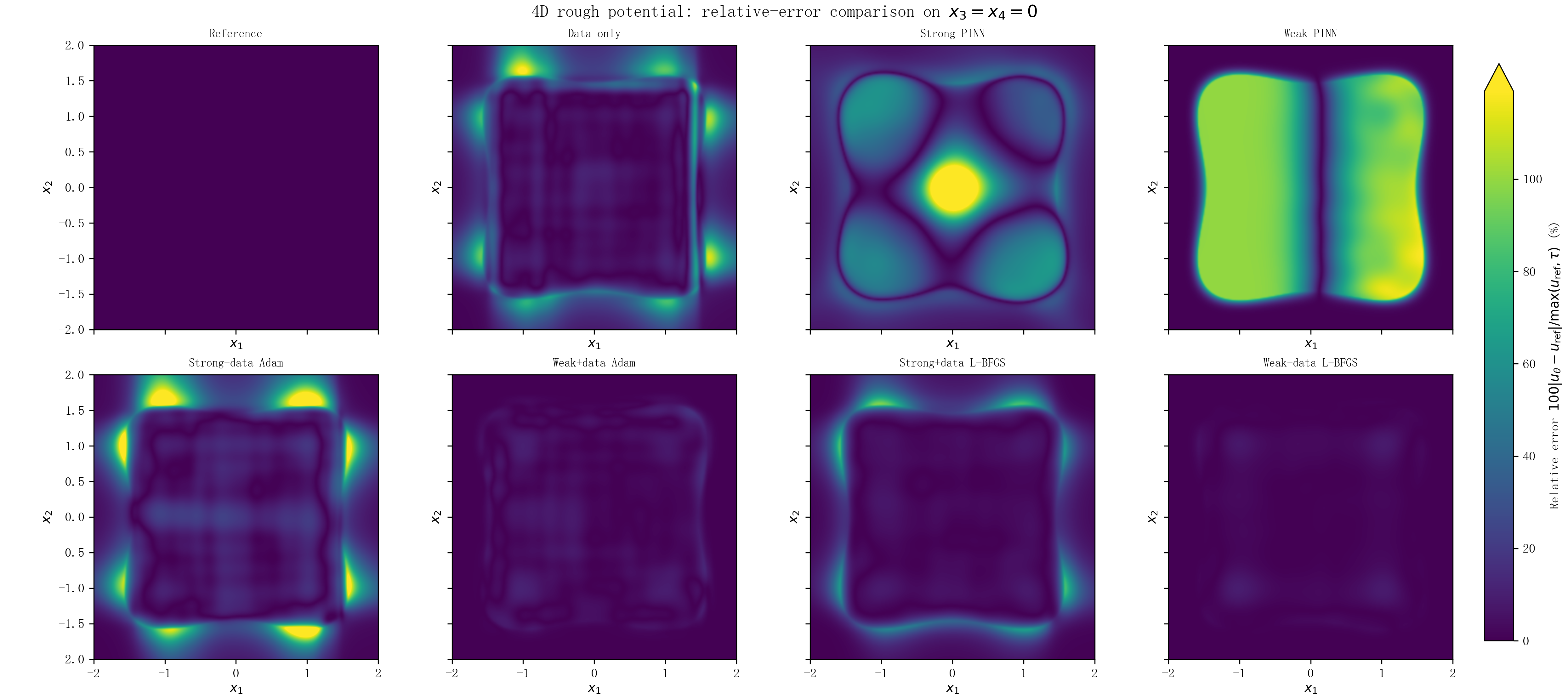}
		\caption{Four-dimensional rough-potential: pointwise relative-error distributions of the seven methods on the slice \(x_3=x_4=0\).}
		\label{fig:rough4d-rel}
		\end{figure}
		
		\begin{figure}[!htbp]
		\centering
		\safeincludegraphics[width=0.86\textwidth,keepaspectratio]{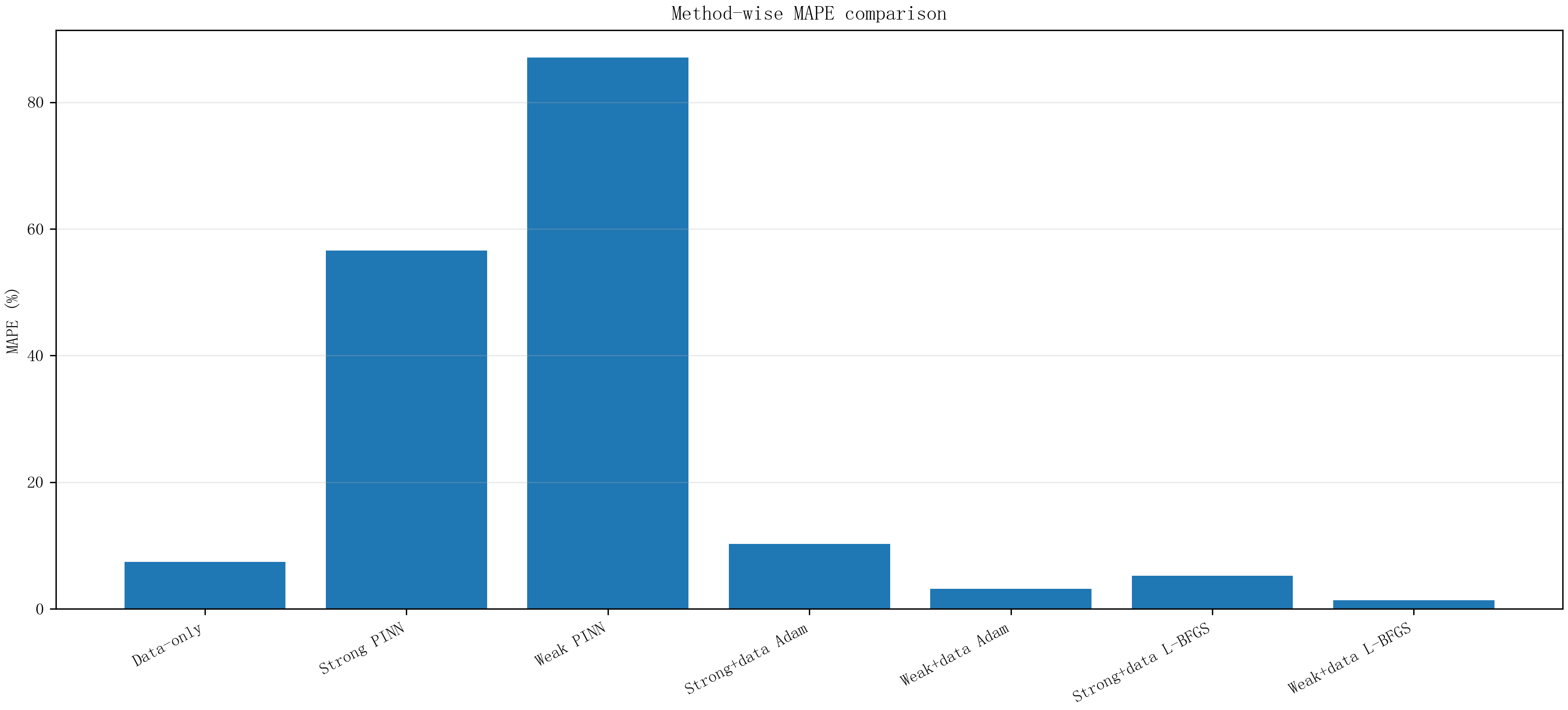}
		\caption{Four-dimensional rough-potential: effective-region MAPE comparison of the seven methods.}
		\label{fig:rough4d-mape}
		\end{figure}

		\FloatBarrier
		\Needspace{8\baselineskip}
		\subsection{Rough Single-Well Potential in Eight Dimensions}
		
		\paragraph{Model and Experimental Setup.}
		The eight-dimensional potential is
        \begin{equation}
        V(\bm x)
        =
        \frac12\sum_{i=1}^{8}\lambda_i x_i^2
                +\varepsilon^4\sum_{j=1}^{8}
        \prod_{m=0}^{3}
        \sin\left(\frac{2\pi x_{j+m}}{\varepsilon}\right),
        \label{eq:8d_potential}
        \end{equation}
        where the coordinate indices are taken cyclically, so that \(x_{i+8}=x_i\). We set \(\varepsilon=0.05\) and \(\lambda_i=2.0+0.1(i-1)\). The drift is \(\bm b=-\nabla V\), with \(D=0.125\), on the computational domain \([-0.7,0.7]^8\). The MC histogram contains \(7{,}337{,}767\) samples in the computational domain, from which \(10{,}000\) reference bins are selected. Training uses \(6{,}000\) physics points and \(1{,}500\) monitoring points. 
		
		The network uses the log-density parameterization introduced in Section~\ref{sec:method}. Physics and monitoring points are sampled from the truncated-Gaussian--uniform mixture in Algorithm~\ref{alg:importance}, with importance correction according to ~\eqref{eq:importance_weight_discrete}--\eqref{eq:importance_quadrature_discrete}. Stage C uses a subset of the existing MC reference bins and reuses the Stage B physics and monitoring points.

		\begin{figure}[!htbp]
		\centering
		\safeincludegraphics[width=0.98\textwidth,height=0.43\textheight,keepaspectratio]{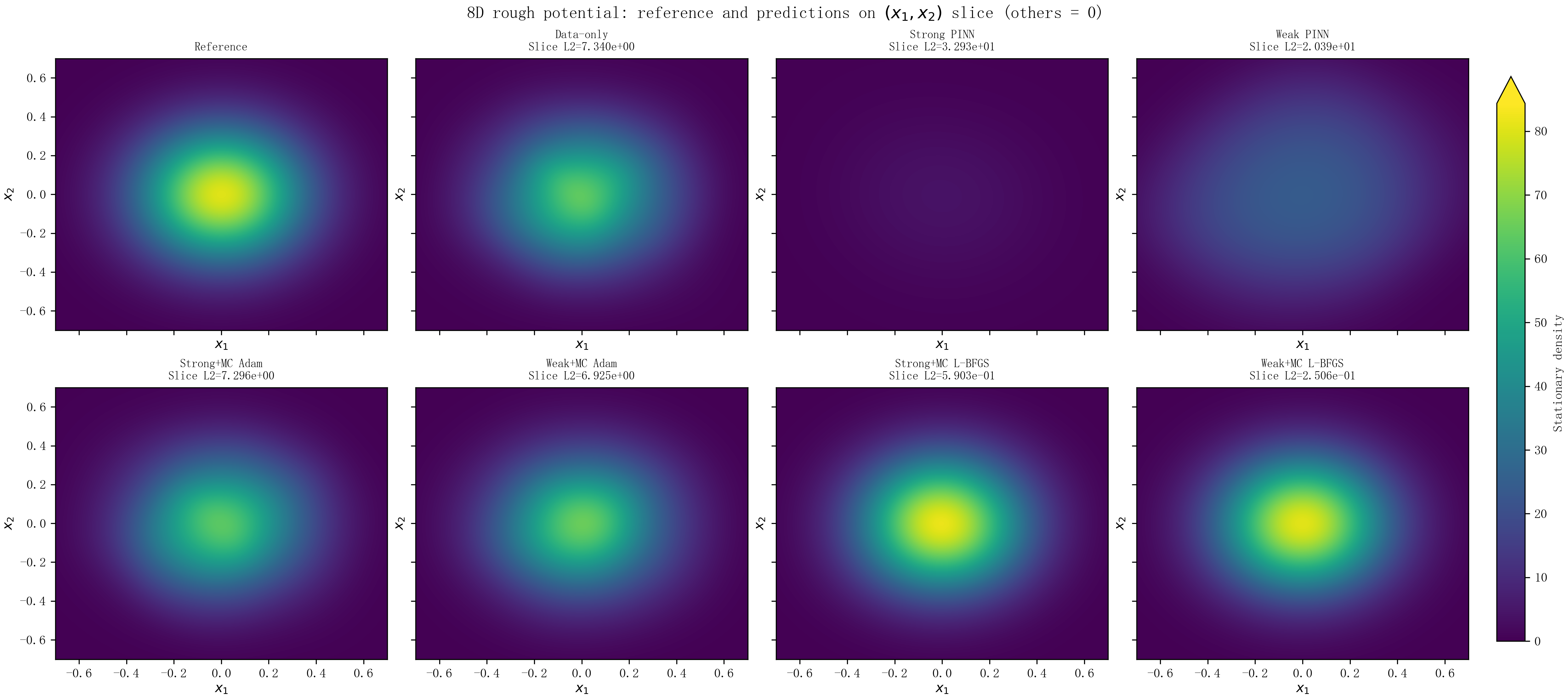}
		\caption{Rough single-well potential: reference stationary density and predictions of the seven methods on the $(x_1,x_2)$ slice with all remaining coordinates set to zero.  }
		\label{fig:rough8d-density}
		\end{figure}
		
		\paragraph{Results and Discussion}
				On the \((x_1,x_2)\) slice with the remaining coordinates fixed at zero, the reference density, predictions, and errors are shown in Figures~\ref{fig:rough8d-density}--\ref{fig:rough8d-rel}; the MAPE comparison is summarized in Figure~\ref{fig:rough8d-mape}. Weak+MC+L-BFGS achieves the lowest effective-region MAPE of \(1.2406\%\), compared with \(4.2581\%\) for Strong+MC+L-BFGS. Along the weak-form training procedure, the MAPE decreases from \(19.6102\%\) for Data-only to \(10.7225\%\) after Stage B and further to \(1.2406\%\) after L-BFGS fine-tuning.

		\begin{figure}[!htbp]
		\centering
		\safeincludegraphics[width=0.98\textwidth,height=0.43\textheight,keepaspectratio]{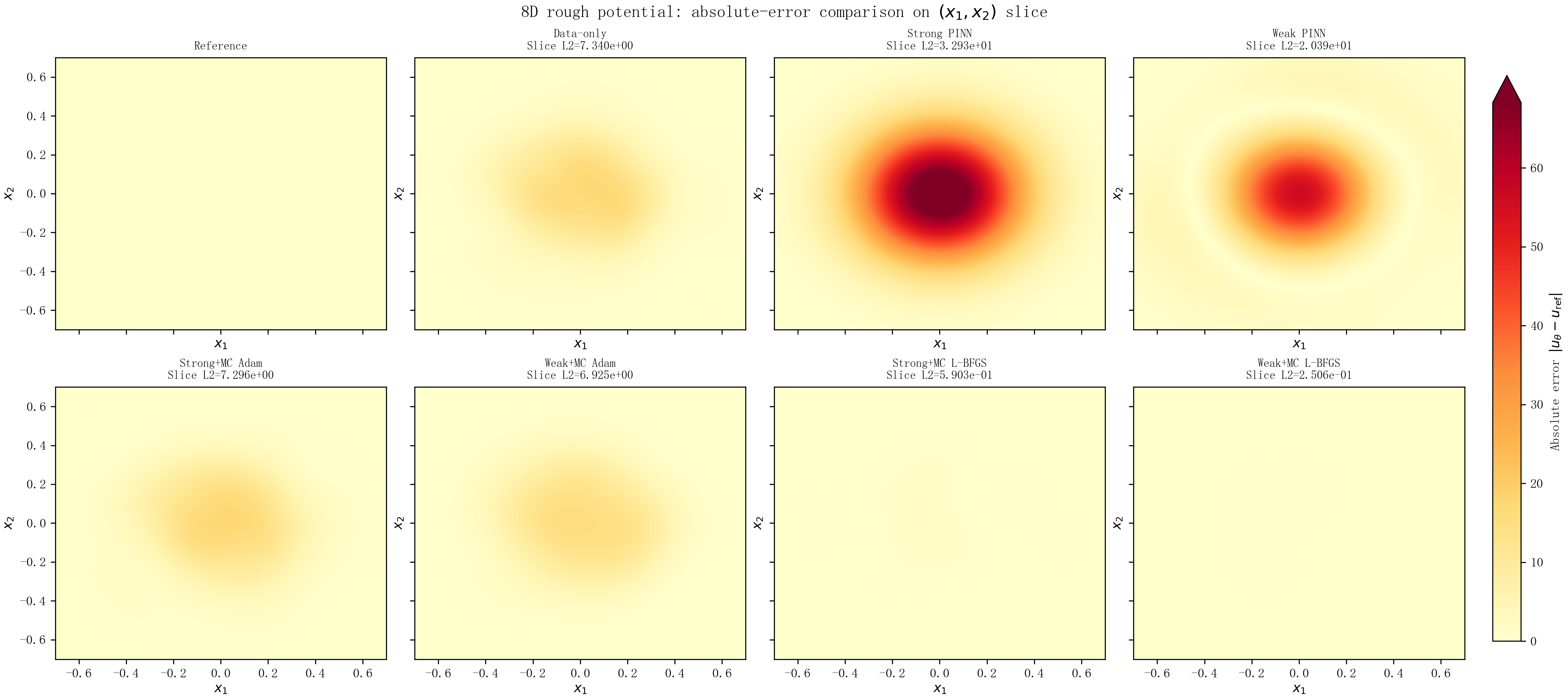}
		\caption{Eight-dimensional rough single-well potential: absolute-error distributions of the seven methods on the \((x_1,x_2)\) slice with all remaining coordinates set to zero.}
		\label{fig:rough8d-abs}
		\end{figure}
		
		\begin{figure}[!htbp]
		\centering
		\safeincludegraphics[width=0.98\textwidth,height=0.43\textheight,keepaspectratio]{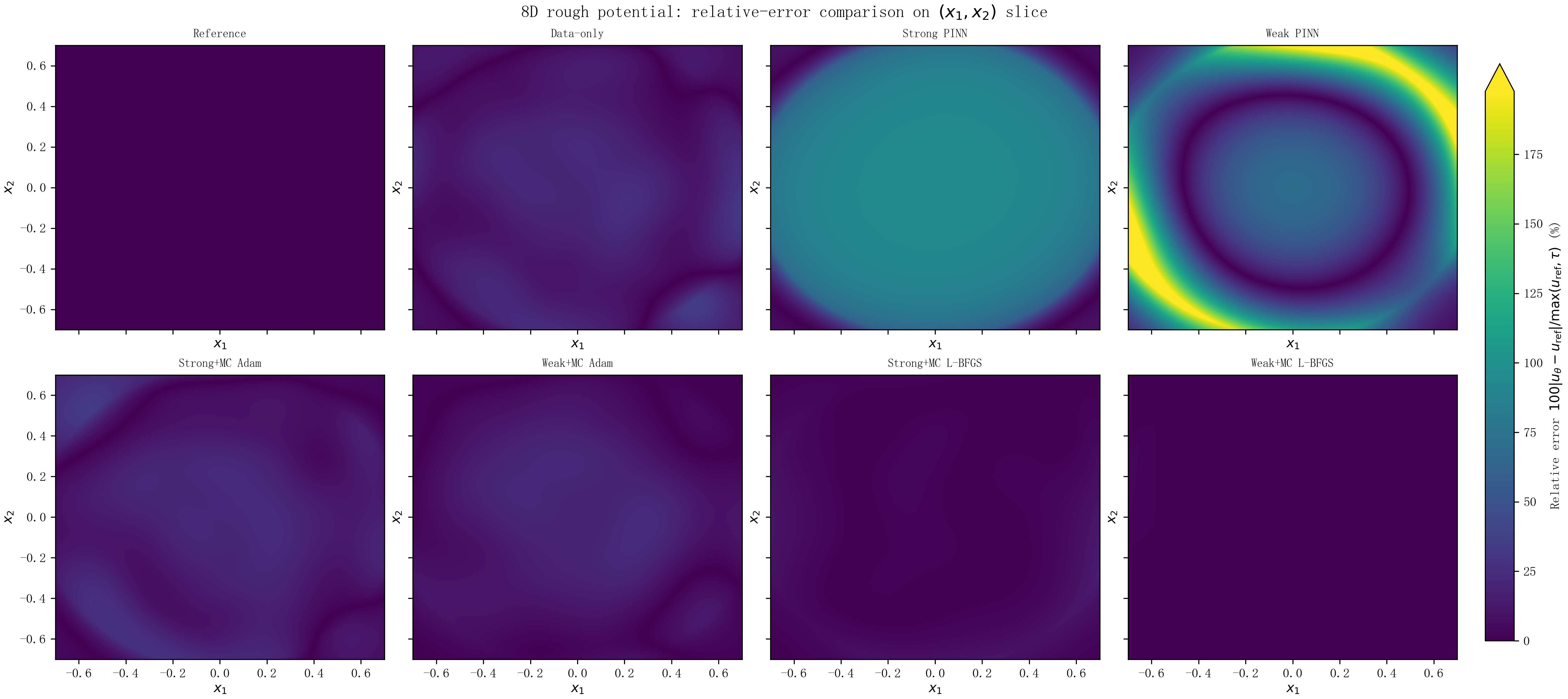}
		\caption{Eight-dimensional rough single-well potential: pointwise relative-error distributions of the seven methods on the \((x_1,x_2)\) slice with all remaining coordinates set to zero.}
		\label{fig:rough8d-rel}
		\end{figure}
		
		\begin{figure}[!htbp]
		\centering
		\safeincludegraphics[width=0.86\textwidth,keepaspectratio]{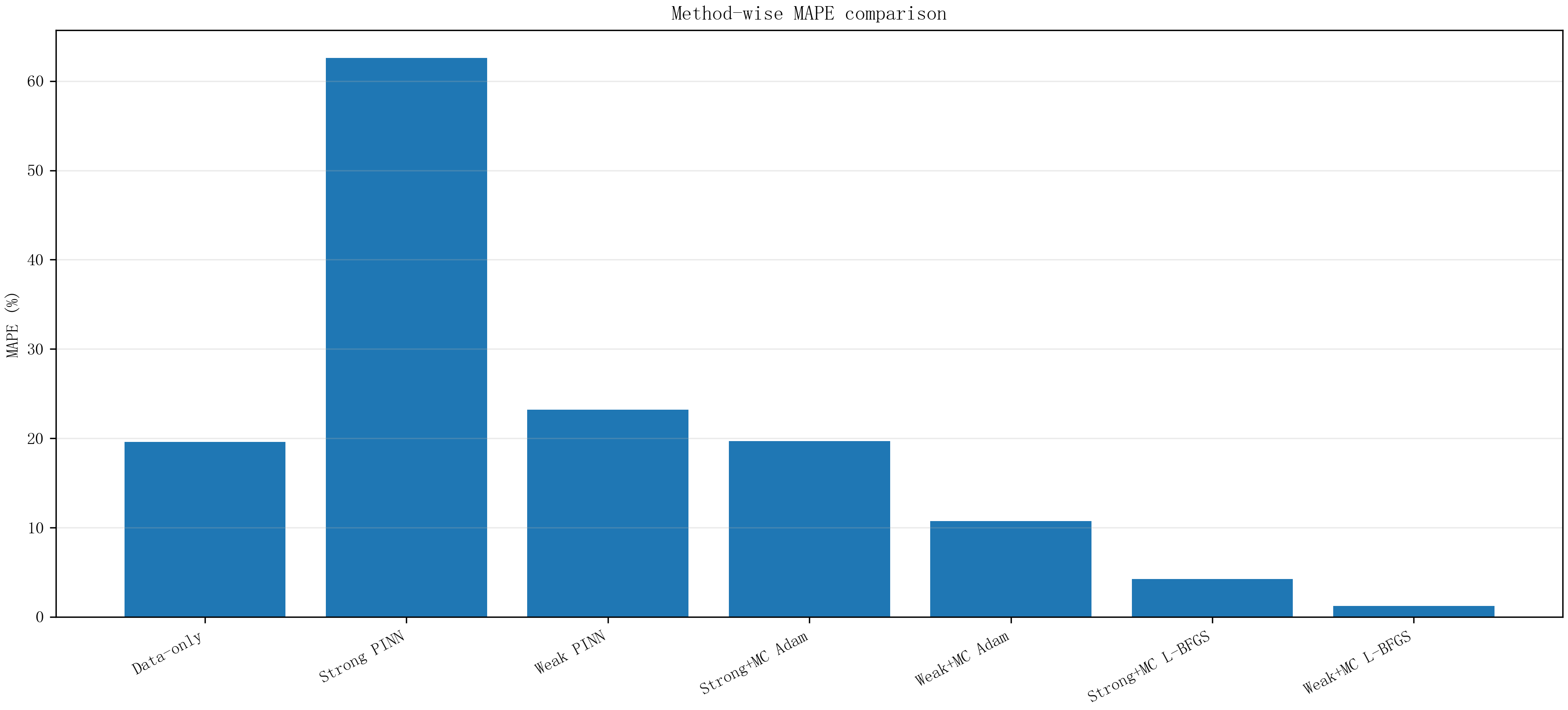}
		\caption{Eight-dimensional rough single-well potential: effective-region MAPE comparison of the seven methods.}
		\label{fig:rough8d-mape}
		\end{figure}

		\FloatBarrier
		\Needspace{8\baselineskip}
		\subsection{Two-Dimensional Equilibrium Positive-Feedback System}
		\label{subsec:positive_feedback}
				\paragraph{Model and Experimental Setup.}
		Consider the equilibrium positive-feedback gene-regulatory system
		\begin{align}
		b_1(x_1,x_2)
		&=
		a+\beta\frac{x_1^n}{S^n+x_1^n}
		-\gamma x_1+\kappa x_2,\\
		b_2(x_1,x_2)
		&=
		a+\beta\frac{x_2^n}{S^n+x_2^n}
		-\gamma x_2+\kappa x_1.
		\end{align}
		Here \(a\) is the basal production rate, \(\beta\) is the positive-feedback strength, \(\gamma\) is the degradation rate, \(\kappa\) is the symmetric coupling strength, \(n\) is the Hill coefficient, and \(S\) is the half-saturation constant. We set
		\begin{equation}
		a=0.05,\qquad
		\beta=0.75,\qquad
		\gamma=1,\qquad
		\kappa=0.20,\qquad
		n=4,\qquad
		S=0.5.
		\end{equation}
		
		The noise variance is \(\sigma^2=0.04\), so \(D=0.02\). The drift admits a scalar potential \(\Phi\) satisfying
        \begin{equation}
        \nabla\Phi(\bm x)=\bm b(\bm x).
        \end{equation}
 The constant of proportionality is evaluated independently by Sobol quadrature on the computational domain and is used only to construct final evaluation. The MC reference is generated using \(8000\) Euler-Maruyama trajectories with time step \(\Delta t=0.002\). After a burn-in time of \(30\), states are sampled every \(0.05\) during the subsequent \(10\) time units, yielding \(1.6\times10^6\) stationary samples. The computational domain is \([-0.6,1.6]^2\), and \(2000\) MC reference points are used for training. The network uses the same low-dimensional architecture as the preceding examples. The Stage B monitoring score uses data and physics weights of \(0.10:1\). Stage C starts from the selected Stage B checkpoint and reuses the same physics points together with the \(2000\) MC reference points; the Stage C data and physics weights are \(0.03:1\).
		
		\begin{figure}[!htbp]
		\centering
		\safeincludegraphics[width=0.98\textwidth,height=0.43\textheight,keepaspectratio]{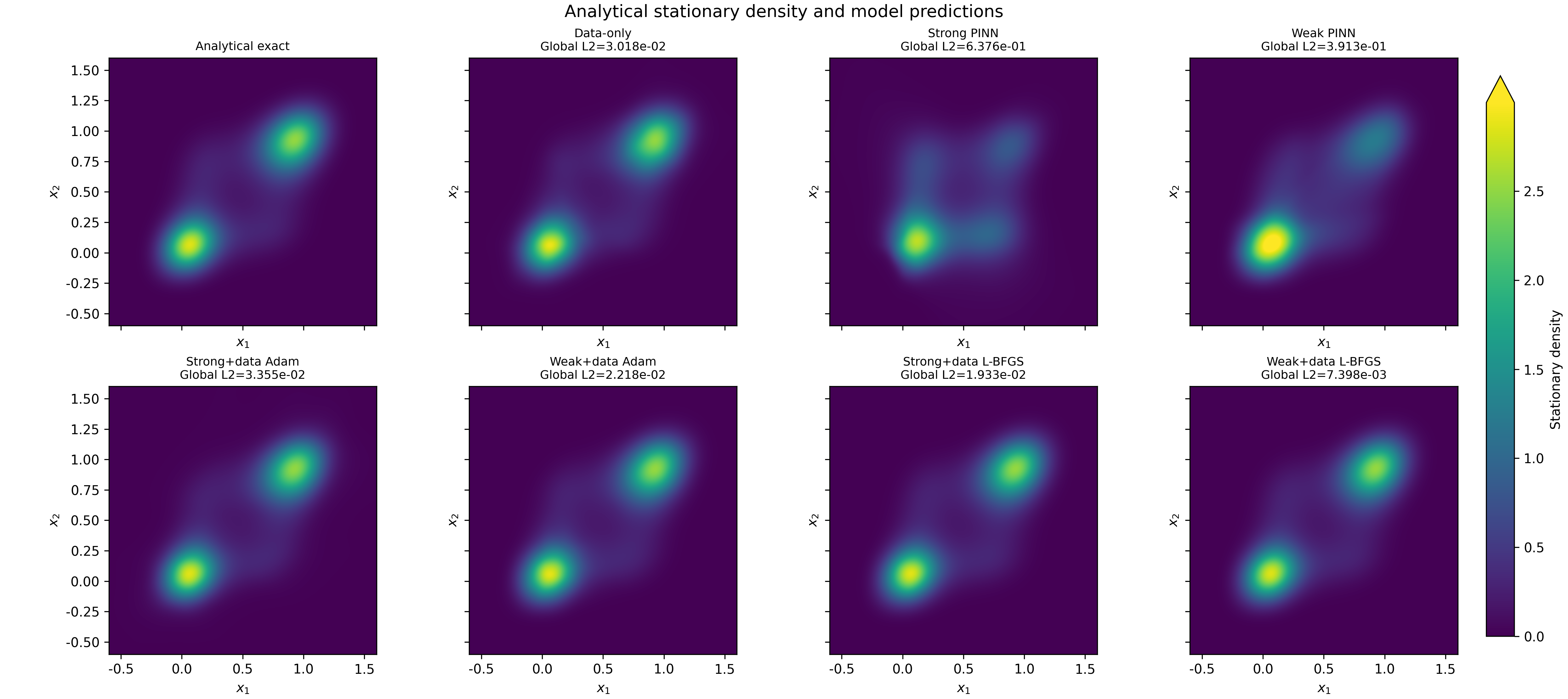}
		\caption{Two-dimensional equilibrium positive-feedback system: analytical stationary density and full-domain predictions of the seven methods.}
		\label{fig:fp2d-density}
		\end{figure}
		
		\paragraph{Results and Discussion}
		The analytical stationary density and the predictions of the seven methods are shown in Figures~\ref{fig:fp2d-density}--\ref{fig:fp2d-rel}; the MAPE comparison is summarized in Figure~\ref{fig:fp2d-mape}.
		
		Weak+MC+L-BFGS achieves the lowest effective-region MAPE of \(0.7518\%\), compared with \(3.2612\%\) for Strong+MC+L-BFGS. For the weak-form hybrid method, the MAPE decreases from \(5.3825\%\) for Data-only to \(3.4041\%\) after Stage B and further to \(0.7518\%\) after L-BFGS fine-tuning. The result shows that the same framework remains effective for this equilibrium gene-regulatory model.

		\begin{figure}[!htbp]
		\centering
		\safeincludegraphics[width=0.98\textwidth,height=0.43\textheight,keepaspectratio]{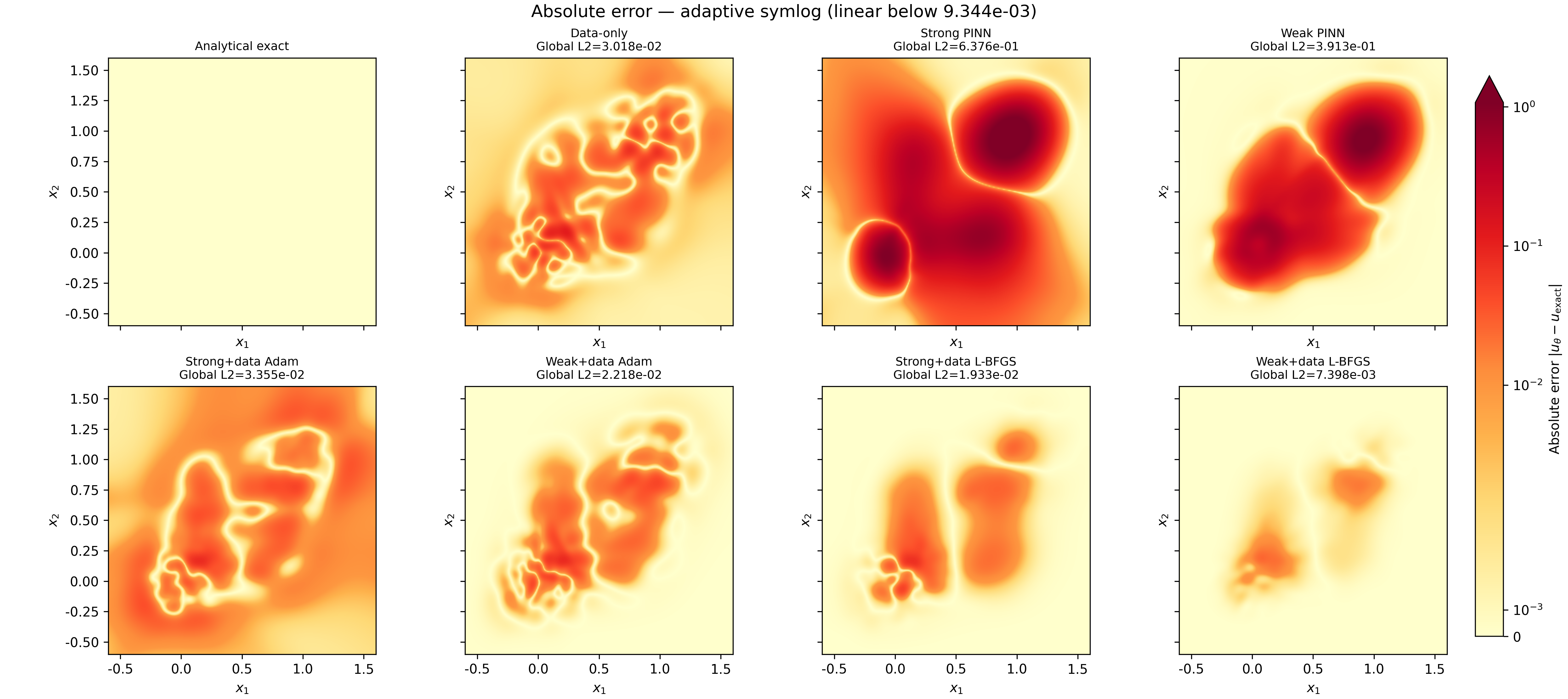}
		\caption{Two-dimensional equilibrium positive-feedback system: full-domain absolute-error distributions of the seven methods.}
		\label{fig:fp2d-abs}
		\end{figure}
		
		\begin{figure}[!htbp]
		\centering
		\safeincludegraphics[width=0.98\textwidth,height=0.43\textheight,keepaspectratio]{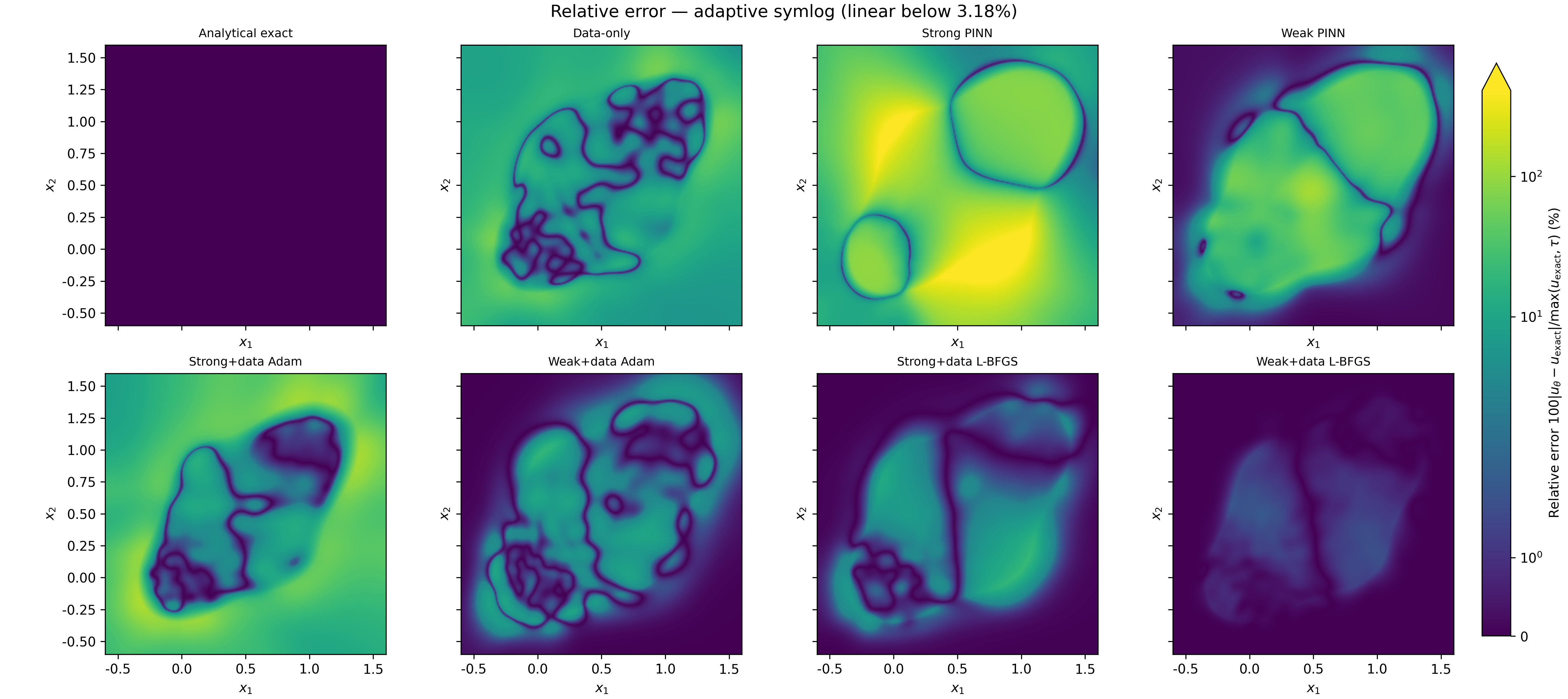}
		\caption{Two-dimensional equilibrium positive-feedback system: full-domain pointwise relative-error distributions of the seven methods, evaluated against the analytical reference density.}
		\label{fig:fp2d-rel}
		\end{figure}
		
		\begin{figure}[!htbp]
		\centering
		\safeincludegraphics[width=0.86\textwidth,keepaspectratio]{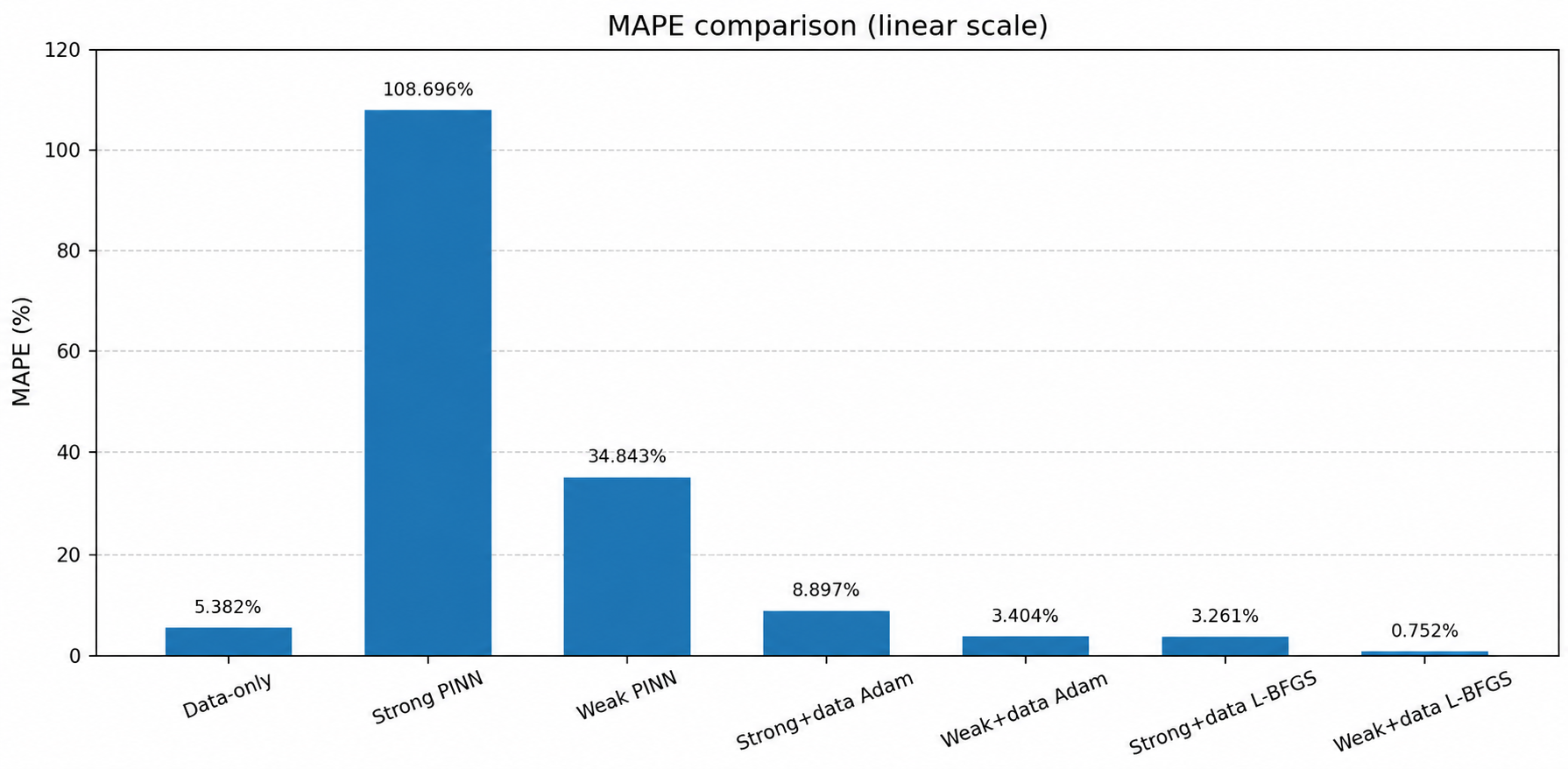}
		\caption{Two-dimensional equilibrium positive-feedback system: effective-region MAPE comparison of the seven methods.}
		\label{fig:fp2d-mape}
		\end{figure}

		\FloatBarrier
		\Needspace{8\baselineskip}
		\subsection{Cross-Dimensional Summary and Ablation Studies}
		\label{subsec:cross_summary}
		
		Table~\ref{tab:main_summary} summarizes the effective-region MAPE for the six benchmark examples considered above.

		\begin{table}[H]
\centering
\caption{Effective-region MAPE (\%) for the six main benchmarks.}
\label{tab:main_summary}
\resizebox{\textwidth}{!}{%
\begin{tabular}{lcccccc}
\toprule
Method & 2D double well & 2D rough rotation & 4D rough & 4D nonequilibrium & 8D rough & 2D positive feedback\\
\midrule
Data-only & 7.2102 & 11.3059 & 8.2217 & 10.3307 & 19.6102 & 5.3825\\
Strong PINN & 262.8252 & 163.5147 & 106.1517 & 65.4960 & 62.5901 & 108.6956\\
Weak/Integral PINN & 47.8723 & 97.4698 & 2.6324 & 57.5599 & 23.2246 & 34.8433\\
Strong + MC Adam & 8.9731 & 13.6343 & 8.2758 & 10.0236 & 19.6914 & 8.8971\\
Weak/Integral + MC Adam & 1.2187 & 2.4990 & 3.5404 & 8.3098 & 10.7225 & 3.4041\\
Strong + MC + L-BFGS & 1.8557 & 3.1981 & 6.3159 & 7.8237 & 4.2581 & 3.2612\\
\textbf{Weak/Integral + MC + L-BFGS} & \textbf{0.3452} & \textbf{0.4813} & \textbf{1.7430} & \textbf{4.2620} & \textbf{1.2406} & \textbf{0.7518}\\
\bottomrule
\end{tabular}}
\end{table}
		
		Weak/Integral+MC+L-BFGS gives the lowest MAPE in all six examples. The rough rotational example also demonstrates the pointwise-orthogonal non-gradient extension developed in Section~\ref{sec:formulation}. The stage ablation shows the same pattern across all six examples: Weak/Integral+MC (Adam) improves upon Data-only, and the subsequent L-BFGS fine-tuning further reduces the MAPE.

		\FloatBarrier
        \Needspace{8\baselineskip}
        
        \subsection{Robustness to Noisy MC Data}
        \label{subsec:noise_robustness}

        To assess robustness to noisy MC data, we compare Strong+MC+L-BFGS and Weak+MC+L-BFGS. The two methods use the same MC data and differ in the form of the physics constraint. The effective-region MAPE in \eqref{eq:mape} is used for evaluation.

        \paragraph{Noise Model.}
        Let \(\{(x_i,\widehat m_i)\}_{i=1}^{N_{\rm ref}}\) denote the MC reference data. For a noise level \(\alpha\in[0,1)\), we define
        \begin{equation}
        \widetilde m_i^{(\alpha)}
        =
        \widehat m_i\left(1+\alpha\xi_i\right),
        \qquad
        \xi_i\overset{\rm i.i.d.}{\sim}\mathcal U(-1,1),
        \label{eq:mc_multiplicative_noise}
        \end{equation}
        where $\mathcal U(-1,1)$ is the normal distribution.  Equivalently, we generate \(U_i\sim\mathcal U(0,1)\) and  set \(\xi_i=2U_i-1\).
        Thus, the relative perturbation of each nonzero MC density value is bounded by \(\alpha\). Only the MC density values are perturbed; the state coordinates, SDE parameters  and physics points are unchanged. For each example, the same realization \(\{\xi_i\}\) is reused across the tested noise levels, and the two methods use the same perturbed MC data and training-point sets.

        The three two-dimensional examples use \(\alpha=0.1,0.3,0.4\), while the four- and eight-dimensional examples use \(\alpha=0.05,0.1,0.2\). For the positive-feedback example, the noise tests use \(1600\) MC reference points and \(800\) physics points; both methods use the same reduced training sets.

        \begin{table}[H]
        \centering
        \caption{Effective-region MAPE (\%) for Strong+MC+L-BFGS and Weak+MC+L-BFGS on the three two-dimensional examples with noisy MC data. Boldface indicates the lower value under each noise condition.}
        \label{tab:noise_2d_mape}
        \resizebox{\textwidth}{!}{
        \begin{tabular}{lccccccccc}
        \toprule
        & \multicolumn{3}{c}{2D double well}
        & \multicolumn{3}{c}{2D rough rotation}
        & \multicolumn{3}{c}{2D positive feedback}\\
        \cmidrule(lr){2-4}
        \cmidrule(lr){5-7}
        \cmidrule(lr){8-10}
        Method
        & \(\alpha=0.1\) & \(\alpha=0.3\) & \(\alpha=0.4\)
        & \(\alpha=0.1\) & \(\alpha=0.3\) & \(\alpha=0.4\)
        & \(\alpha=0.1\) & \(\alpha=0.3\) & \(\alpha=0.4\)\\
        \midrule
        Strong + MC + L-BFGS
        & 2.0200 & 2.8670 & 3.3819
        & 4.4411 & 5.2529 & 5.0861
        & 2.1462 & 15.7072 & 8.8053\\
        \textbf{Weak + MC + L-BFGS}
        & \textbf{0.4204} & \textbf{0.8915} & \textbf{1.1188}
        & \textbf{0.4604} & \textbf{0.9366} & \textbf{1.1810}
        & \textbf{0.5189} & \textbf{1.2454} & \textbf{3.3559}\\
        \bottomrule
        \end{tabular}}
        \end{table}

         Table~\ref{tab:noise_2d_mape} shows that Weak+MC+L-BFGS achieves a lower MAPE than Strong+MC+L-BFGS in all nine two-dimensional noise tests. In the rough rotational example, the weak-form MAPE increases from \(0.4604\%\) at \(\alpha=0.1\) to \(1.1810\%\) at \(\alpha=0.4\), while the corresponding strong-form errors remain above \(4\%\) throughout. The positive-feedback example shows a similar robustness advantage, with a particularly large difference at \(\alpha=0.3\).

        \begin{table}[H]
        \centering
        \caption{Effective-region MAPE (\%) for Strong+MC+L-BFGS and Weak+MC+L-BFGS on the four- and eight-dimensional rough-potential examples with noisy MC data. Boldface indicates the lower value under each noise condition.}
        \label{tab:noise_highdim_mape}
        \resizebox{\textwidth}{!}{
        \begin{tabular}{lcccccc}
        \toprule
        & \multicolumn{3}{c}{4D rough-potential} & \multicolumn{3}{c}{8D rough-potential}\\
        \cmidrule(lr){2-4}\cmidrule(lr){5-7}
        Method & $\alpha=0.05$ & $\alpha=0.1$ & $\alpha=0.2$
        & $\alpha=0.05$ & $\alpha=0.1$ & $\alpha=0.2$\\
        \midrule
        Strong + MC + L-BFGS
        & 6.8912 & 8.8394 & 12.1804
        & 4.7434 & 8.1804 & 35.9194\\
        \textbf{Weak + MC + L-BFGS}
        & \textbf{1.9946} & \textbf{2.2216} & \textbf{2.5418}
        & \textbf{1.3418} & \textbf{1.7620} & \textbf{3.2789}\\
        \bottomrule
        \end{tabular}}
        \end{table}

        The same comparison in four and eight dimensions is reported in Table~\ref{tab:noise_highdim_mape}. At \(\alpha=0.2\), Weak+MC+L-BFGS gives MAPE values of \(2.5418\%\) and \(3.2789\%\), respectively, compared with \(12.1804\%\) and \(35.9194\%\) for Strong+MC+L-BFGS.

        \paragraph{Robustness to Rough Potentials.}
We vary the roughness scale $\varepsilon$ while keeping the
test-function settings fixed. For each $\varepsilon$, we compare
the Strong and Weak methods against a Monte Carlo reference
computed for the corresponding potential.  The WeakFamily+MC+L-BFGS method has the lower MAPE at
all four values.

\begin{table}[H]
\centering
\caption{Effective-region MAPE (\%) under four roughness scales in the four-dimensional nonequilibrium experiment.}
\label{tab:ness4d_epsilon_sweep}
\resizebox{\textwidth}{!}{%
\begin{tabular}{ccccc}
\toprule
$\varepsilon$ & Strong + MC + L-BFGS & WeakFamily + MC + L-BFGS & Strong $-$ Weak (pp) & Conclusion\\
\midrule
0.10 (default) & 7.8237 & 4.2620 & 3.5617 & Weak lower\\
0.15 & 7.6576 & 3.9668 & 3.6908 & Weak lower\\
0.20 & 9.6262 & 4.0909 & 5.5353 & Weak lower\\
0.25 & 8.2183 & 4.3362 & 3.8821 & Weak lower\\
\bottomrule
\end{tabular}}
\end{table}

\paragraph{Robustness to Density Estimation Errors.}
To isolate the effect of density perturbations, we keep the coordinates,
drift, MC labels  and physics evaluation points fixed. We replace $u_\theta$ by $\widetilde u_\theta$ only when evaluating the physics loss, where $\widetilde u_\theta(x)
  =u_\theta(x)\left[1+\alpha\varepsilon\,\xi(x/\varepsilon)\right].$  Here $\xi$ is a fixed smooth function used in both the Strong and Weak methods. The strong loss involves second derivatives of $\widetilde u_\theta$,
whereas the weak loss involves only first derivatives.
For all $\alpha>0$, we normalize the physics loss using the same
scale as in the clean case: $s_{p}^{(B)}(\alpha)=s_{p}^{(B)}(0)$, $s_{p}^{(C)}(\alpha)=s_{p}^{(C)}(0).$
We use the same normalization at all noise levels so that the physics
losses can be compared on a common scale.

\begin{table}[H]
\centering
\caption{Four-dimensional $u_\theta$ perturbations with fixed normalization.}
\label{tab:ness4d_utheta_noise}
\resizebox{\textwidth}{!}{%
\begin{tabular}{ccccc}
\toprule
$\alpha$ & Strong MAPE (\%) & Weak MAPE (\%) & Strong shift (\%) & Weak shift (\%)\\
\midrule
0.0000 & 8.061 & 4.201 & -- & --\\
0.0010 & 7.664 & 4.169 & 1.313 & 0.203\\
0.0025 & 9.856 & 4.183 & 2.489 & 0.186\\
0.0050 & 9.820 & 4.117 & 2.514 & 0.291\\
0.0100 & 9.643 & 4.152 & 2.610 & 0.221\\
0.0150 & 9.649 & 4.211 & 2.611 & 0.158\\
0.0200 & 21.995 & 4.177 & 12.198 & 0.174\\
0.0250 & 27.973 & 4.155 & 16.793 & 0.252\\
\bottomrule
\end{tabular}}
\end{table}

We use test-function weights $\omega_k\propto\lambda_k^{0.15}$,
normalized to have mean one. Without perturbations, this choice gives
a weak MAPE of $4.201\%$, the lower of the two settings tested.
The seven-method comparison above uses equal weights.  We keep the normalization of the physics loss fixed across noise levels.
Relative to the unperturbed solutions, the weak solution changes by less
than $0.30\%$ at all tested noise levels, whereas the strong solution
changes by $16.793\%$ at $\alpha=0.025$. At this noise level, the MAPE
is $4.155\%$ for weak and $27.973\%$ for strong.
These results show that, in this experiment, the weak method, which
uses only first derivatives, is less sensitive to perturbations in
$u_\theta$ than the strong method, which uses second derivatives.


Overall, Weak+MC+L-BFGS achieves a lower MAPE in all $15$ tested noise
conditions, suggesting that the variational formulation is more robust
to noise in the MC data.

		\FloatBarrier
		\section{Conclusion}
        \label{sec:conclusion}
        This work develops a three-stage learning framework for stationary Fokker--Planck equations by combining Monte Carlo density data with a variational physics constraint.  The MC data used for pretraining provide weak supervision, while the variational loss enforces consistency with the underlying physical laws. For equilibrium steady states, the loss is derived from the free-energy dissipation structure, and the formulation is further extended to a class of non-equilibrium steady states with a pointwise orthogonal drift decomposition.

        Numerical experiments demonstrate the effectiveness of the proposed framework across low- and high-dimensional problems, rough potentials, and noisy Monte Carlo observations. In particular, the weak/variational hybrid method consistently improves upon the corresponding strong-form formulation in the reported noise tests, while the four- and eight-dimensional examples show the performance of the variational formulation in higher dimensions. This example illustrates the application of the variational formulation
to a stationary problem with non-gradient drift.

       Our method applies to equilibrium steady states and to nonequilibrium
steady states satisfying the pointwise orthogonality condition.
Using a family of test functions, we also treat stationary problems with
non-gradient drift that does not satisfy the pointwise orthogonality condition. The four-dimensional example demonstrates this extension:
the stationary current is nonzero, and the drift is not pointwise
orthogonal to the potential gradient.   Several questions remain for future study, including the extension of  the method to time-dependent problems.

\appendix

\section{Convergence of the Weak--MC Method}
\label{app:theory}

\subsection{Problem formulation}

Consider the Fokker--Planck equation on a bounded 
Lipschitz domain $\Omega\subset\mathbb R^d$, with diffusion
coefficient $D>0$ and potential $V\in C^1(\overline\Omega)$.
The flux is $\mathbf J=-D\nabla u-u\nabla V$, and the
equation reads $\partial_t u+\nabla\cdot\mathbf J=0$.
We seek an equilibrium satisfying $\mathbf J=0$ and
$\int_\Omega u\,\mathrm dx=1$. Its density is
$u^*=Z^{-1}e^{-V/D}$, where
$Z=\int_\Omega e^{-V/D}\,\mathrm dx$.

For a positive candidate $u\in C^1(\overline\Omega)$, define
\begin{align}
 \Phi(u)
 &=
 \frac{\displaystyle\int_\Omega
       |\mathbf J|^2/u_\varepsilon\,\mathrm dx}
      {\displaystyle\int_\Omega u\,\mathrm dx},
 \qquad
 u_\varepsilon=\max\{u,\varepsilon_u\},
 \quad \varepsilon_u=10^{-12},
 \label{app:weak-continuous}\\
 \mathcal L(u)
 &=
 \lambda_{\rm W}\Phi(u)
 +\lambda_{\rm MC}\mathcal L_{\rm MC}(u),
 \qquad
 \mathcal L_{\rm MC}(u)
 =
 \frac1M\sum_{j=1}^M
 \bigl(u(y_j)-v_j\bigr)^2,
 \label{app:full-loss}
\end{align}
where $\lambda_{\rm W},\lambda_{\rm MC}>0$ are fixed
and $v_j$ are the MC reference values at $y_j$.
Since $\Omega$ is connected, $\Phi(u)=0$ implies
$u=c\,u^*$ for some $c>0$. Conversely,
$\Phi(c\,u^*)=0$ for every $c>0$.
The MC term constrains the constant $c$ within this family.
The integrals are approximated by equally weighted
Sobol quadrature in 2D and 4D and by quadrature with
importance-sampling weights in 8D. Both procedures
approximate the functional $\Phi$.
The analysis below concerns local minimizers of
$\mathcal L$ with fixed weights and does not establish
convergence of Adam or L-BFGS.

\subsection{Discrete loss and local error estimates}
In this subsection, we use a local quadratic approximation
of the combined loss to quantify the reduction in mean-square
error obtained by adding the weak loss to the MC data-fitting term.

The vector $u_h\in(0,\infty)^{m_h}$ contains the density
values at the quadrature points, and $R_hu_h$ approximates
the flux $\mathbf J$, where $R_h$ is a fixed linear map.
The vector $a_h$ contains the positive quadrature weights.
The diagonal matrix $W_h(u_h)$ contains the entries
$(a_h)_j/\max\{(u_h)_j,\varepsilon_u\}$, repeated for
each component of the flux. The discrete loss is
\begin{equation}\label{app:weak-discrete}
 \Phi_h(u_h)
 =
 \frac{(R_hu_h)^\mathsf T W_h(u_h)(R_hu_h)}
      {a_h^\mathsf T u_h}.
\end{equation}

Assume that a positive discrete equilibrium $u_h^*$
satisfies $R_hu_h^*=0$ and
$(u_h^*)_j\neq\varepsilon_u$ for every $j$.
Then $\Phi_h$ is smooth near $u_h^*$, and
\begin{equation}\label{app:weak-hessian}
 \nabla\Phi_h(u_h^*)=0,
 \qquad
 \nabla^2\Phi_h(u_h^*)
 =
 \frac{2R_h^\mathsf T W_h(u_h^*)R_h}
      {a_h^\mathsf T u_h^*}.
\end{equation}
Indeed, every term involving a derivative of $W_h$
or of the mass denominator contains a factor
$R_hu_h^*$ and therefore vanishes at $u_h^*$.

Define the symmetric positive semidefinite matrix
\begin{equation}\label{app:Kh}
 K_h
 =
 \frac{m_h}{2}\nabla^2\Phi_h(u_h^*)
 =
 \frac{m_h}{a_h^\mathsf T u_h^*}
 R_h^\mathsf T W_h(u_h^*)R_h.
\end{equation}
For this finite-dimensional analysis, assume that the
MC data observe every component of $u_h$ with equal
weights, so that $M=m_h$ and the data loss is
$m_h^{-1}\|u_h-v_h\|_2^2$.
Writing $v_h=u_h^*+e_h$ and $z=u_h-u_h^*$,
the quadratic approximation of the combined loss,
multiplied by $m_h$, is
\begin{equation}\label{app:quadratic-model}
 F_h(z;e_h)
 =
 \lambda_{\rm W}z^\mathsf T K_hz
 +\lambda_{\rm MC}\|z-e_h\|_2^2.
\end{equation}
The first term penalizes deviations from the discrete
equilibrium equation, while the second fits the MC data.

\begin{theorem}[Mean-square error estimate]
\label{thm:spectral}
Assume $\lambda_{\rm W},\lambda_{\rm MC}>0$,
$\mathbb E[e_h]=0$, and
$\operatorname{Cov}(e_h)=\Sigma_h$ with
$0<\operatorname{tr}(\Sigma_h)<\infty$.
The unique minimizer of \eqref{app:quadratic-model} is
\begin{equation}\label{app:filter}
 \bar z_h=B_he_h,
 \qquad
 B_h
 =
 \lambda_{\rm MC}
 \bigl(\lambda_{\rm MC}I+\lambda_{\rm W}K_h\bigr)^{-1}.
\end{equation}
Its mean-square error satisfies
\begin{equation}\label{app:general-ratio}
 \frac{\mathbb E\|\bar z_h\|_2^2}
      {\mathbb E\|e_h\|_2^2}
 =
 \frac{\operatorname{tr}
       (B_h\Sigma_hB_h^\mathsf T)}
      {\operatorname{tr}(\Sigma_h)}
 \leq 1.
\end{equation}
If $\Sigma_h=\zeta_h^2I$ with $\zeta_h>0$, then
\begin{equation}\label{app:isotropic-ratio}
 \frac{\mathbb E\|\bar z_h\|_2^2}
      {\mathbb E\|e_h\|_2^2}
 =
 \frac1{m_h}\sum_{k=1}^{m_h}
 \left(
 \frac{\lambda_{\rm MC}}
      {\lambda_{\rm MC}+\lambda_{\rm W}\mu_{h,k}}
 \right)^2,
\end{equation}
where $\mu_{h,1},\ldots,\mu_{h,m_h}$ are the eigenvalues
of $K_h$. Error components corresponding to positive
eigenvalues are reduced, whereas those corresponding
to zero eigenvalues remain unchanged.
\end{theorem}

\begin{proof}
The stationary equation for \eqref{app:quadratic-model} is
\[
 \bigl(\lambda_{\rm MC}I+\lambda_{\rm W}K_h\bigr)z
 =
 \lambda_{\rm MC}e_h.
\]
The matrix on the left is positive definite,
so $F_h$ is strictly convex and its unique minimizer
is given by \eqref{app:filter}.

Since $\mathbb E[e_h]=0$, we have
$\mathbb E[\bar z_h]=0$ and
$\operatorname{Cov}(\bar z_h)
=B_h\Sigma_hB_h^\mathsf T$.
Taking traces gives the identity in
\eqref{app:general-ratio}.
Moreover, the eigenvalues of $B_h$ are
$\lambda_{\rm MC}/
(\lambda_{\rm MC}+\lambda_{\rm W}\mu_{h,k})$,
which belong to $(0,1]$. Hence
$\|B_he_h\|_2\leq\|e_h\|_2$, proving the inequality.

If $\Sigma_h=\zeta_h^2I$, then
$\mathbb E\|\bar z_h\|_2^2
=\zeta_h^2\operatorname{tr}(B_h^2)$ and
$\mathbb E\|e_h\|_2^2=m_h\zeta_h^2$.
Summing the squared eigenvalues of $B_h$ gives
\eqref{app:isotropic-ratio}.
\end{proof}
 Theorem~\ref{thm:spectral} establishes the mean-square
error estimate
$\mathbb E\|\bar z_h\|_2^2\leq\mathbb E\|e_h\|_2^2$
for the minimizer of \eqref{app:quadratic-model},
showing that incorporating the weak loss into the local
quadratic model does not increase the mean-square error
relative to the original MC data.
\begin{remark}
Since $W_h(u_h^*)$ is positive definite,
$\ker K_h=\ker R_h$. If the MC error satisfies
$R_he_h=0$, then \eqref{app:filter} gives
$\bar z_h=e_h$. Thus, adding the weak loss does
not reduce such errors.

In particular, $\Phi_h(cu_h^*)=0$ for every $c>0$,
so the weak loss cannot determine the multiplicative
constant $c$. This constant is determined by the
MC data-fitting term. If
$\ker R_h=\operatorname{span}\{u_h^*\}$,
all errors satisfying $R_he_h=0$ are scalar
multiples of $u_h^*$.

When $\Sigma_h=\zeta_h^2I$, the mean-square error
ratio tends to zero precisely when
\[
  \frac1{m_h}\sum_{k=1}^{m_h}
  \left(
    \frac{\lambda_{\rm MC}}
    {\lambda_{\rm MC}+\lambda_{\rm W}\mu_{h,k}}
  \right)^2
  \longrightarrow0.
\]
Each zero eigenvalue contributes $1/m_h$ to this
sum. Therefore, $\dim(\ker R_h)/m_h\to0$ is
a necessary condition.
\end{remark}

\subsection{Mean-square error estimate for the nonlinear loss}

We now prove that the mean-square error bound in
Theorem~\ref{thm:spectral} also holds for positive
local minimizers of the original discrete loss.

With $v_h=u_h^*+e_h$, let $\widehat z_h$ be a local
minimizer of $m_h\lambda_{\rm W}\Phi_h(u_h^*+z)
  +\lambda_{\rm MC}\|z-e_h\|_2^2$
over $u_h^*+z>0$.
For $0<t\leq1$, positivity of $u_h^*$ gives
$u_h^*+t\widehat z_h
=(1-t)u_h^*+t(u_h^*+\widehat z_h)
\geq t(u_h^*+\widehat z_h)$ componentwise.
Hence
\[
  \max\{(u_h^*+t\widehat z_h)_j,\varepsilon_u\}
  \geq
  t\max\{(u_h^*+\widehat z_h)_j,\varepsilon_u\}
\]
for every component $j$.
By the definition of $W_h$ and positivity of $a_h$,
\[
  W_h(u_h^*+t\widehat z_h)
  \preceq t^{-1}W_h(u_h^*+\widehat z_h),
  \qquad
  a_h^{\mathsf T}(u_h^*+t\widehat z_h)
  \geq t\,a_h^{\mathsf T}(u_h^*+\widehat z_h).
\]
Since $R_hu_h^*=0$, it follows that
\begin{equation}
\begin{aligned}
  \Phi_h(u_h^*+t\widehat z_h)
  &=
  \frac{
    t^2(R_h\widehat z_h)^{\mathsf T}
    W_h(u_h^*+t\widehat z_h)
    (R_h\widehat z_h)}
    {a_h^{\mathsf T}(u_h^*+t\widehat z_h)}
  \\
  &\leq \Phi_h(u_h^*+\widehat z_h).
\end{aligned}
\label{app:nonlinear-monotonicity}
\end{equation}
Thus, interpolation toward the discrete equilibrium
$u_h^*$ does not increase the weak loss.

By local minimality and
\eqref{app:nonlinear-monotonicity}, for $t<1$
sufficiently close to $1$, $\|\widehat z_h-e_h\|_2^2
  \leq \|t\widehat z_h-e_h\|_2^2.$
Expanding and dividing by $1-t>0$ yields
$(1+t)\|\widehat z_h\|_2^2
\leq2e_h^{\mathsf T}\widehat z_h$.
Letting $t\uparrow1$ and applying the
Cauchy--Schwarz inequality, we obtain $\|\widehat z_h\|_2^2
  \leq e_h^{\mathsf T}\widehat z_h
  \leq\|e_h\|_2\|\widehat z_h\|_2.$
Consequently, $\|\widehat z_h\|_2\leq\|e_h\|_2$.
For random MC errors with
$\mathbb E\|e_h\|_2^2<\infty$, taking expectations, we have
\begin{equation*}
  \mathbb E\|\widehat z_h\|_2^2
  \leq\mathbb E\|e_h\|_2^2.
\end{equation*}

Therefore, incorporating the weak loss does not
increase the mean-square error at any positive
local minimizer of the original discrete loss.
This extends the mean-square error bound of
Theorem~\ref{thm:spectral} to the nonlinear problem.
\subsection{Local error analysis of the neural-network approximation}

We establish a mean-square error estimate for the
local quadratic model of the neural-network loss.  Let $\Phi_P(\theta)$ denote the weak loss evaluated
at the fixed quadrature points.
Suppose that $u_{\theta^*}=u^*$ and that the network
outputs and $\Phi_P$ are twice continuously
differentiable near $\theta^*$.
Let $J_D$ be the Jacobian of
$(u_\theta(y_1),\ldots,u_\theta(y_M))^{\mathsf T}$
at $\theta^*$, and set
$H_W=(M/2)\nabla_\theta^2\Phi_P(\theta^*)$.
Since $\Phi_P\geq0$ and $\Phi_P(\theta^*)=0$,
we have $\nabla_\theta\Phi_P(\theta^*)=0$ and
$H_W\succeq0$.

Define
$e=(v_1-u^*(y_1),\ldots,v_M-u^*(y_M))^{\mathsf T}$
as the MC error.
Linearizing the network outputs at $\theta^*$
and retaining the quadratic term of the weak los, we obtain
, after multiplying the loss by $M$,
\begin{equation}
  Q(\delta\theta;e)
  =
  \lambda_{\rm W}\delta\theta^{\mathsf T}H_W\delta\theta
  +
  \lambda_{\rm MC}\|J_D\delta\theta-e\|_2^2.
  \label{app:network-quadratic}
\end{equation}

At any minimizer $\delta\theta$ of
\eqref{app:network-quadratic}, we have the following first-order
optimality condition
\[
  \bigl(
    \lambda_{\rm W}H_W
    +\lambda_{\rm MC}J_D^{\mathsf T}J_D
  \bigr)\delta\theta
  =
  \lambda_{\rm MC}J_D^{\mathsf T}e.
\]
Taking the scalar product with $\delta\theta$, one obtains
\[
  \lambda_{\rm W}\delta\theta^{\mathsf T}H_W\delta\theta
  +\lambda_{\rm MC}\|J_D\delta\theta\|_2^2
  =
  \lambda_{\rm MC}e^{\mathsf T}J_D\delta\theta.
\]
Since $H_W\succeq0$, the Cauchy--Schwarz inequality
implies
\[
  \|J_D\delta\theta\|_2^2
  \leq e^{\mathsf T}J_D\delta\theta
  \leq\|e\|_2\|J_D\delta\theta\|_2.
\]
Hence $\|J_D\delta\theta\|_2\leq\|e\|_2$.
For random MC errors with
$\mathbb E\|e\|_2^2<\infty$, we obtain
\begin{equation}
  \mathbb E\|J_D\delta\theta\|_2^2
  \leq\mathbb E\|e\|_2^2.
  \label{app:network-error}
\end{equation}
Thus, every minimizer of the local quadratic model
satisfies the mean-square error bound
\eqref{app:network-error}, without requiring
$J_D$ to have full column rank.

In Stage C, the MC data and quadrature points
from Stage B are fixed, and L-BFGS is applied
to minimize
\[
  \lambda_{\rm W}\Phi_P(\theta)
  +
  \frac{\lambda_{\rm MC}}{M}
  \sum_{j=1}^M
  \bigl(u_\theta(y_j)-v_j\bigr)^2.
\]
The estimate \eqref{app:network-error} extends the
preceding mean-square error bound to the linearized
neural-network approximation, whose error at the
observation points is $J_D\delta\theta$.

\subsection{Uniform convergence of the quadrature approximation}\label{app:quadrature-convergence}

We prove that the quadrature approximation of the weak
loss converges uniformly on compact parameter sets.  For a positive candidate $u$, define
\begin{equation}
  \Phi_N(u)
  =
  \frac{\displaystyle
    \sum_{i=1}^N w_i
    \frac{|\mathbf J(X_i)|^2}{u_\varepsilon(X_i)}}
  {\displaystyle
    \sum_{i=1}^N w_i u(X_i)},
  \qquad
  \mathbf J=-D\nabla u-u\nabla V,
  \label{app:empirical}
\end{equation}
where $u_\varepsilon=\max\{u,\varepsilon_u\}$.
Let $\Theta$ be a compact parameter set.
Suppose that $u_\theta>0$ and that $u_\theta$
and $\nabla u_\theta$ are jointly continuous
on $\overline\Omega\times\Theta$.
Positivity and compactness then imply
\begin{equation}
  \inf_{\theta\in\Theta}
  \int_\Omega u_\theta(x)\,\mathrm dx>0.
  \label{app:quadrature-mass}
\end{equation}

For Sobol quadrature, take $w_i=1$.
On a rectangular domain, the points are obtained
by affine rescaling of a Sobol sequence.
Its uniform distribution gives
\begin{equation}
  \frac{|\Omega|}{N}\sum_{i=1}^N f(X_i)
  \longrightarrow
  \int_\Omega f(x)\,\mathrm dx
  \qquad
  \bigl(f\in C(\overline\Omega)\bigr);
  \label{app:sobol-consistency}
\end{equation}
see \cite{Sobol1967}.
For a general bounded Lipschitz domain, the same
conclusion holds when Sobol points are generated
in a containing box and only points in $\Omega$
are retained. Indeed, the boundary has zero
Lebesgue measure, so uniform distribution applies
both to the indicator of $\Omega$ and to the
extension of $f$ by zero outside $\Omega$.

For importance sampling, let $X_i$ be independent
samples from a fixed probability density $q$ on
$\Omega$, with $q>0$ almost everywhere, and take
$w_i=1/q(X_i)$.
For every $f\in C(\overline\Omega)$,
\[
  \mathbb E\left[\frac{|f(X_i)|}{q(X_i)}\right]
  =\int_\Omega|f(x)|\,\mathrm dx<\infty.
\]
The strong law of large numbers therefore yields
\begin{equation}
  \frac1N\sum_{i=1}^N w_i f(X_i)
  \longrightarrow
  \int_\Omega f(x)\,\mathrm dx
  \quad\text{almost surely}.
  \label{app:importance-consistency}
\end{equation}
In particular, taking $f=1$ gives
$N^{-1}\sum_{i=1}^Nw_i\to|\Omega|$ almost surely.
The normalization factors in
\eqref{app:sobol-consistency} and
\eqref{app:importance-consistency} cancel in
the ratio \eqref{app:empirical}.

We next prove uniformity with respect to $\theta$.
Since $V\in C^1(\overline\Omega)$ and
$\varepsilon_u>0$, the numerator and denominator
integrands in \eqref{app:empirical}, with
$u=u_\theta$, depend continuously on $\theta$
in $C(\overline\Omega)$.
Their two families are therefore compact
in the supremum norm.  For each $\varepsilon=2^{-k}$, choose finitely
many integrands that approximate each family
with error at most $\varepsilon$ in this norm.
Apply \eqref{app:sobol-consistency} or
\eqref{app:importance-consistency} to these
finite collections.
The error caused by replacing an integrand
with its approximation is at most
$|\Omega|\varepsilon$ for the normalized
Sobol sum and
$\varepsilon N^{-1}\sum_{i=1}^Nw_i$
for the importance-sampling sum.
The corresponding integral error is at most
$|\Omega|\varepsilon$.
Letting $N\to\infty$ and then $k\to\infty$ yields
uniform convergence in $\theta\in\Theta$ of the
normalized numerator and denominator of
$\Phi_N(u_\theta)$ in \eqref{app:empirical}.
For importance sampling, only countably many
integrands are used in this argument, so the
uniform convergence holds almost surely.
By \eqref{app:quadrature-mass}, we have the normalized
denominators of  $\Phi_N(u_\theta)$ in \eqref{app:empirical} are uniformly bounded away from
zero for all sufficiently large $N$.
Passing to the quotient  in \eqref{app:empirical} yields
\begin{equation*}
  \sup_{\theta\in\Theta}
  \bigl|\Phi_N(u_\theta)-\Phi(u_\theta)\bigr|
  \rightarrow0.
\end{equation*}
Therefore, the Sobol and importance-sampling
approximations converge to the weak loss
\eqref{app:weak-continuous}, uniformly over
$\theta\in\Theta$, with almost sure convergence
in the importance-sampling case.

\subsection{Conclusion}

The weak loss is zero exactly when $u=c\,u^*$,
$c>0$. The MC data are used to determine the
constant $c$.  For the discrete problem, we prove
$\mathbb E\|\widehat z_h\|_2^2
\leq\mathbb E\|e_h\|_2^2$
for every positive local minimizer of the
combined loss. Thus, its mean-square error
does not exceed that of the MC data.
For the quadratic model,
Theorem~\ref{thm:spectral} also exhibits the
exact ratio between these two errors.

For neural networks, the error bound is proved
only for the local quadratic approximation.
The analysis does not prove convergence of
L-BFGS or the same bound for the network
computed by the algorithm.  Finally, \eqref{app:quadrature-convergence}
shows that the weighted sums used in computation
converge to the integrals defining $\Phi$,
uniformly for $\theta\in\Theta$, as the number
of quadrature points increases.

        \bibliographystyle{plainurl}
        \bibliography{ref}
		\end{document}